\documentclass[journal]{IEEEtran}
\title{Revisiting the Misspecified Cram$\acute{\text{e}}$r-Rao Bound}
\author{Malaak Khatib, \IEEEmembership{Student Member, IEEE}, 
Nadav Harel, \IEEEmembership{Member, IEEE}, 
Joseph Tabrikian, \IEEEmembership{Fellow, IEEE}, 
and
Tirza Routtenberg, \IEEEmembership{Senior Member, IEEE}
\thanks{\footnotesize{The authors are with the School of Electrical and Computer Engineering, Ben-Gurion University of the Negev, Beer-Sheva 84105, Israel, e-mail: \{malaak;nadavhar\}@post.bgu.ac.il, \{joseph; tirzar\}@bgu.ac.il. This research was supported by the ISRAEL SCIENCE FOUNDATION (Grant No. 1148/22 and  2493/23).  
M. Khatib is a fellow of the AdR Women Doctoral Program.
 }
}\vspace{-0.5cm}}

\usepackage{amssymb,amsmath,rangecite,amsthm}
\newtheorem{theorem}{Theorem}
\newtheorem{definition}{Definition}
\newtheorem{claim}{Claim}
\newtheorem{proposition}{Proposition}
\newtheorem{conjecture}{Conjecture}
\usepackage[dvips]{color}
\usepackage{lscape}

\newcommand{\avec}{{\bf{a}}}

\newcommand{\xvec}{{\bf{x}}}

\newcommand{\vvec}{{\bf{v}}}
\newcommand{\gvec}{{\bf{g}}}

\newcommand{\hvec}{{\bf{h}}}

\newcommand{\etavec}{{\bf{\eta}}}

\newcommand{\onevec}{{\bf{1}}}
\newcommand{\zerovec}{{\bf{0}}}

\newcommand{\Amat}{{\bf{A}}}
\newcommand{\Bmat}{{\bf{B}}}

\newcommand{\Jmat}{{\bf{J}}}
\newcommand{\Imat}{{\bf{I}}}

\newcommand{\Wmat}{{\bf{W}}}

\newcommand{\Sigmamat}{{\bf{\Sigma}}}
\newcommand{\Lambdamat}{{\bf{\Lambda}}}
\renewcommand{\epsilon}{\varepsilon}
\newcommand{\ud}{\,\mathrm{d}}

\newcommand{\Tr}{{\rm Tr}}

\newcommand{\define}{\stackrel{\triangle}{=}}

\def\bsigma{{\mbox{\boldmath $\Sigma$}}}

\def\alphavecsmall{{\mbox{\boldmath {\scriptsize $\alpha$}}}}

\def\alphavecmicro{{\mbox{\boldmath \scalebox{0.95}{$\alpha$}}}}  

\def\lambdavec{{\mbox{\boldmath $\lambda$}}}

\def\gammavec{{\mbox{\boldmath $\gamma$}}}
\def\gammavecsmall{{\mbox{\boldmath {\scriptsize $\gamma$}}}}

\def\gammavecmicro{{\mbox{\boldmath \scalebox{0.95}{$\gamma$}}}}  

\def\etavec{{\mbox{\boldmath $\eta$}}}

\def\etavecmicro{{\mbox{\boldmath \scalebox{0.95}{$\eta$}}}}  

\def\phivec{{\mbox{\boldmath $\phi$}}}

\def\phivecmicro{{\mbox{\boldmath \scalebox{0.95}{$\phi$}}}}  

\def\thetavec{{\mbox{\boldmath $\theta$}}}

\def\thetavecsmall{{\mbox{\boldmath {\scriptsize $\theta$}}}}
\def\thetavectiny{{\mbox{\boldmath {\tiny $\theta$}}}}
\def\thetavecmicro{{\mbox{\boldmath \scalebox{0.95}{$\theta$}}}}  

\def\varthetavec{{\mbox{\boldmath $\vartheta$}}}
\def\varthetavecsmall{{\mbox{\boldmath {\scriptsize $\vartheta$}}}}

\def\varthetavecmicro{{\mbox{\boldmath \scalebox{0.95}{$\vartheta$}}}}  

\def\muvec{{\mbox{\boldmath $\mu$}}}

\def\thetavecsmall{{\mbox{\boldmath {\scriptsize $\theta$}}}}

\def\alphavecsmall{{\mbox{\boldmath {\scriptsize $\alpha$}}}}

\newcommand{\be}{\begin{equation}}
\newcommand{\ee}{\end{equation}}
\newcommand{\beqna}{\begin{eqnarray}}
\newcommand{\eeqna}{\end{eqnarray}}

\usepackage[bookmarks,colorlinks]{hyperref}
\usepackage[most]{tcolorbox}
\usepackage{subcaption}
\usepackage{cite}
\usepackage{amsfonts}
\usepackage{algorithmic}
\usepackage{lipsum}
\usepackage{enumerate, mathtools, float}
\usepackage[inline]{enumitem}
\usepackage[ruled,vlined]{algorithm2e}
\usepackage{xcolor,textcomp}
\usepackage{nameref}
\usepackage{tikz}
\usepackage[printonlyused]{acronym}
\usepackage{graphicx} 

\AtBeginDocument{%
  \setlength{\abovedisplayskip}{5pt plus 2pt minus 2pt}%
  \setlength{\belowdisplayskip}{5pt plus 2pt minus 2pt}%
  \setlength{\abovedisplayshortskip}{3pt plus 1pt minus 1pt}%
  \setlength{\belowdisplayshortskip}{3pt plus 1pt minus 1pt}%
  \setlength{\jot}{3pt}
}

\newcounter{myboxcounter}
\renewcommand{\themyboxcounter}{\arabic{myboxcounter}} 

\usepackage{etoolbox} 

\newtcolorbox{mybox}[2][]{%
  enhanced,
  colback=pink!10,
  colframe=brown!30,
  coltitle=darkgray,
  title=Box~\themyboxcounter: #2,
  boxsep=5pt,
  left=-1.5pt, right=-1.5pt, top=5pt, bottom=1pt,
  before skip=0pt,
  after skip=0pt,
  after=\par\nointerlineskip\noindent\vspace{-8pt}, 
  #1
}

\newtcolorbox{mysubbox}[2][]{%
  enhanced,
  colback=pink!10,
  colframe=brown!30,
  coltitle=darkgray,
  title=,
  boxsep=5pt,        
  left=1pt, right=1pt, top=1pt, bottom=0pt,
  #1
}
\AtEndEnvironment{mybox}{\vspace{-5pt}}
\newcounter{boxeq}

\makeatletter
\@addtoreset{boxeq}{myboxcounter}
\makeatother

\newenvironment{boxequation}
  {%
    \refstepcounter{boxeq}%
    \begingroup
    \begin{equation}
  }
  {%
    \end{equation}
    \endgroup
  }
\newenvironment{boxsubequations}
  {\refstepcounter{boxeq}%
   \begin{subequations}
  }
  {\end{subequations}}

   \usepackage{tabu}
	\usepackage{makecell}

\newcommand{\customlabel}[2]{%
    \refstepcounter{equation}%
    \tag{#2}%
    \label{#1}%
}
\newcommand{\customref}[1]{%
    \hyperref[#1]{\text{(C.\ref*{#1})}}%
}

\usetikzlibrary{calc, arrows.meta, intersections, patterns, positioning, shapes.misc, fadings, through,decorations.pathreplacing}
\usetikzlibrary{shapes,arrows}

\acrodef{crb}[CRB]{Cram$\acute{\text{e}}$r-Rao bound}
\acrodef{mcrb}[MCRB]{misspecified \ac{crb}}
\acrodef{fim}[FIM]{Fisher information matrix}
\acrodef{pdf}[PDF]{probability density function}
\acrodef{cdf}[CDF]{cumulative distribution function}
\acrodef{doa}[DOA]{direction-of-arrival}
\acrodef{mspe}[MSPE]{mean-squared-periodic-error}
\acrodef{mse}[MSE]{mean-squared-error}
\acrodef{wrt}[w.r.t.]{with respect to}
\acrodef{snr}[SNR]{signal-to-noise ratio}
\acrodef{kld}[KLD]{Kullback–Leibler divergence}
\acrodef{map}[MAP]{maximum \emph{a-posteriori}}
\acrodef{ml}[ML]{maximum likelihood}
\acrodef{mml}[MML]{misspecified maximum likelihood}
\acrodef{lhs}[l.h.s.]{left hand side}
\acrodef{ula}[ULA]{uniform linear array}

\usepackage[capitalise,noabbrev]{cleveref}

\crefname{subsection}{Subsection}{Subsections}
\Crefname{appendix}{Appendix}{Appendices}
\Crefname{objective}{Objective}{Objectives}
\Crefname{condition}{condition}{conditions}
\Crefname{reg_condition}{regularity Condition}{regularity Conditions}
\Crefname{conjecture}{Conjecture}{Conjectures}

\Crefname{equation}{}{}
\crefname{figure}{Fig.}{Figs.}

\usepackage{silence}
\begin{document}
\maketitle
\nopagebreak
\begin{abstract}
    Estimation under model misspecification arises in many signal processing problems, where the assumed observation model deviates from the true data-generating mechanism due to errors or simplifications. The misspecified Cram$\acute{\text{e}}$r-Rao bound (MCRB) is a widely recognized mean-squared-error (MSE) lower bound for this case, which has originally been used to describe the asymptotic behavior of the misspecified maximum likelihood (MML) estimator. Despite its widespread use, the MCRB lacks a rigorous characterization of the class of estimators for which it is valid.
    In this paper, we revisit the theory of parameter estimation under model misspecification and re-examine the foundations of the MCRB.
    We first demonstrate these limitations and examine a naive version of the MCRB, which relies only on local misspecified unbiasedness. We show that this bound is generally not tight and may be unattainable. To obtain a meaningful bound, we develop a new derivation based on the concept of pointwise equivalent models. By maximizing the naive bound for these models, we recover the classical MCRB, now supported by a constructive derivation, an explicit characterization of the associated estimator class, and an equality condition.
    This formulation establishes a formal link between local unbiasedness conditions and achievable bounds, offering new insights into the MCRB structure and its relevance to practical estimators. 
    Finally, we define the notion of an efficient misspecified estimator and show that if it exists, it is achieved by the MML estimator. 
\end{abstract}
%
\vspace{-0.25cm}
\section{Introduction}
\label{sec:intro}
\vspace{-0.05cm}
    Performance bounds play a central role in estimation theory by characterizing the achievable \ac{mse}, providing benchmarks for practical estimators, serving as valuable tools for system design, and guiding the design of cognitive and adaptive systems in various applications \cite{Huleihel2013Optimal,Bell2015Cognitive,Greco2018Cognitive}. In classical settings, the \ac{crb} is the most widely used benchmark, mainly due to its simplicity and its asymptotic relation with the \ac{ml} estimator. 
    However, the \ac{crb} relies on the assumption that the statistical/observation model is \emph{correctly specified}.  This assumption is often not valid in practice due to modeling errors, unknown effects, or deliberate model simplifications.  In such cases, estimators operate under an assumed misspecified model, while their bias and \ac{mse} are determined by the true data-generating mechanism,  rendering classical performance bounds, such as the \ac{crb}, inapplicable. Thus, there is a need for rigorous performance bounds that account for misspecifications for system design and performance analysis.


    The asymptotic behavior of the \ac{ml} estimator under model misspecification, commonly referred to as the \ac{mml} estimator, has been extensively studied in the statistical literature since the seminal works of \cite{huber1967behavior,Akaike1998,white1982maximum}. In this setting, the \ac{mml}  estimator maximizes the assumed likelihood and asymptotically converges to the pseudo-true parameter, defined as the minimizer of the \ac{kld} between the true and assumed models. Moreover, its asymptotic \ac{mse} converges to the so-called Huber's sandwich covariance \cite{huber1967behavior}, which is given by the \ac{mcrb}. 
    The \ac{mcrb} has been introduced in \cite{vuong1986cramer} as a lower bound on the \ac{mse} under model misspecification for both parametric and nonparametric models. An alternative derivation of the \ac{mcrb}, based on the covariance form of the Cauchy–Schwarz inequality, has been proposed in \cite{Richmond_Howorits_jour}.  In recent years, the \ac{mcrb} has attracted growing attention within the signal processing community, leading to several important extensions.  These include the constrained \ac{mcrb} \cite{fortunati2016constrained,Richmond_Fusion_2018}, learning-based \ac{mcrb} formulation \cite{habi2023learned}, misspecified Bayesian \acp{crb} \cite{Fortunati_perf_bound,kantor2015prior,richmond2020misspecified,pajovic2018misspecified,tang2023parametric}, and an \ac{mcrb} for quantized observations \cite{rosenthal2025mcrb}. The MCRB has also been employed as a principled tool for model selection \cite{krauz2021composite,rosenthal2025model} and for the analysis of estimation after model selection \cite{harel2023non}. All these developments have enabled the derivation of tractable \acp{mcrb} in a wide range of signal-processing applications, such as in radars \cite{levy2023mcrb,Ren2015Performance,TDOA_2025}, channel estimation \cite{abed2021misspecified,Chlaily_Comon_Jutten2017,Wang2019Misspecified}, integrated sensing and communication \cite{isac},  fake features estimation \cite{hellkvist2023estimation}, \ac{doa} estimation \cite{rosenthal2025mcrb,Khatib_cyclicMCRB}, and time-delay and Doppler estimation \cite{fortunati2024efficiency}.

    Despite these advances, the existing \ac{mcrb} framework exhibits several fundamental limitations that hinder its theoretical clarity and practical applicability. In particular, existing formulations do not explicitly specify the class of estimators for which the \ac{mcrb} constitutes a valid lower bound. In practice, it is often implicitly assumed that the \ac{mcrb} applies to all locally unbiased estimators \ac{wrt} the pseudo-true parameter; however, this condition does not, in general, define the appropriate estimators' class, and as shown in this paper, the \ac{mcrb} fails to bound the \ac{mse} of such estimators.
    In addition, the \ac{mcrb} differs substantially from a naive extension of the classical \ac{crb} to the misspecified setting that relies solely on local misspecified unbiasedness \cite{vuong1986cramer,Richmond_Fusion_2018}. This naive bound is known to be generally not tight and may fail to properly capture the misspecification effect. Yet, a clear characterization of the limitations of the naive \ac{mcrb} and its relation with the \ac{mcrb} remains lacking. 
    Moreover, the original foundational derivation of the \ac{mcrb} in \cite{vuong1986cramer} relies on least-favorable distribution arguments that are not fully transparent from a signal-processing perspective and do not readily extend to other classes of performance bounds, such as the Barankin bound. The state-of-the-art covariance-inequality-based derivation in \cite{Richmond_Howorits_jour} is closely aligned with classical estimation theory, but depends on a first-order expansion of the log-likelihood derivative around the \ac{mml} estimator and its asymptotic properties \cite{white1982maximum}, rendering the resulting bound inherently asymptotic and estimator-specific. Finally, unlike the \ac{crb}, the relationship between the \ac{mcrb} and efficiency under model misspecification and the bound attainability conditions has not been explicitly established.

    In this paper, we present a unified, estimator-agnostic, and non-asymptotic derivation of the \ac{mcrb}
    that explicitly identifies the class of estimators for which the bound is valid. 
    We first discuss the limitations of the existing \ac{mcrb} theory. Then, we introduce a naive \ac{mcrb}, obtained by directly extending classical \ac{crb} arguments under local misspecified unbiasedness, and show that this bound is generally not tight and may be unattainable by estimators operating solely under the assumed model. To resolve these limitations, we introduce the concept of \emph{pointwise equivalent models}, an auxiliary family of \acp{pdf} that coincides with the true model only at the operating point. By maximizing the naive \ac{mcrb} over this family, we recover the tightest achievable lower bound and show that it coincides with the classical \ac{mcrb}. Importantly, this construction yields an explicit characterization of the class of estimators for which the \ac{mcrb} is a valid \ac{mse} lower bound and clarifies the conditions under which the bound is attainable. Finally, we define a principled notion of efficiency under model misspecification and show that under standard regularity conditions, the \ac{mml} estimator is efficient in this generalized sense.

    The main contributions of this paper are as follows:
    \begin{itemize}[leftmargin=*]
    \item \textbf{Limitations of existing MCRB formulations}: 
    We identify fundamental limitations of current \ac{mcrb} formulations, including non-transparent derivations and the lack of a well-defined class of estimators for which the bound holds. We demonstrate via a simple counterexample that commonly used misspecified unbiasedness conditions do not characterize the appropriate estimator class.

    \item \textbf{Analysis of the naive MCRB}: 
    We derive the naive MCRB \cite{vuong1986cramer} by solving a constrained \ac{mse} minimization problem that enforces only local misspecified unbiasedness. 
    We show that this bound is always looser than or equal to the  \ac{mcrb}, reduces to the oracle \ac{crb} when the pseudo-true and true parameters coincide, and is generally unattainable due to its equality condition.

    \item \textbf{New derivation of the \ac{mcrb} via pointwise equivalent models}: 
    We introduce the concept of pointwise equivalent models and show that maximizing the naive misspecified bound over this class recovers the classical \ac{mcrb}. This yields a transparent, estimator-agnostic derivation, together with an explicit equality condition and an estimator-class characterization.

   \item \textbf{Misspecified efficiency}: 
    We introduce a principled notion of efficiency under model misspecification via the equality conditions of the \ac{mcrb}. We show that, under standard regularity conditions, the \ac{mml} estimator is efficient in this sense, establishing a direct analogue of the classical ML--CRB efficiency relationship for misspecified models.

    \end{itemize}
        {\em{Organization:}} The remainder of this paper is structured as follows. 
        In Section \ref{model_and_prob_section}, we provide the background on estimation under model misspecification and the \ac{mcrb}, and discuss the limitations of the existing theory. We introduce the naive \ac{mcrb} and discuss its limitations in Section \ref{naive_MCRB_section}. Next, we introduce the concept of pointwise equivalent models and use it to derive a revised \ac{mcrb} in Section~\ref{sec:equivalent_models}. Building on this framework, we formalize the revised \ac{mcrb} theorem and characterize efficient estimators in Section~\ref{sub:Final_MCRB}. Finally, we conclude the paper in Section~\ref{Conclusions}.

         {\em{Notation:}} Throughout this paper, we denote vectors by boldface lowercase letters and matrices by boldface uppercase letters. The $m$th element of the vector $\avec$ and the $(m,k)$th element of the matrix $\Amat$ are denoted by $a_{m}$ and $[\Amat]_{m,k}$, respectively. The identity matrix of dimensions $K \times K$ is denoted by $\Imat_{K}$, $\zerovec$ denotes a vector/matrix of zeros, and $\onevec$ denotes a vector/matrix of ones. The notations $(\cdot)^{T}$, $(\cdot)^{-1}$ and ${\text{Tr}}(\cdot)$ denote the transpose, inverse, and trace operators, respectively. For any symmetric matrices $\Amat$ and $\Bmat$ of the same dimension, the notation $\Amat \succeq \Bmat$ implies that $\Amat - \Bmat$ is a positive semidefinite matrix. The gradient of a vector function $\gvec$ of $\alphavecmicro$, $\nabla_{\alphavecsmall} \gvec(\alphavecmicro)$, is a matrix in which $[\nabla_{\alphavecsmall} \gvec(\alphavecmicro)]_{m,k} = \frac{\partial g_{m}(\alphavecmicro)}{\partial \alpha_{k}}$. The  Hessian matrix of a scalar function $g$ of $\alphavecmicro$  $\nabla_{\alphavecsmall} \! {\nabla_{\alphavecsmall}}^{\!\!\!\! T} g(\alphavecmicro) = {\nabla_{\alphavecsmall}}^{\!\!\!\! 2} \,g(\alphavecmicro)$. We denote $\nabla_{\alphavecsmall} \gvec(\alphavecmicro) \big|_{\alphavecsmall = \alphavecsmall_{0}} = \nabla_{\alphavecsmall} \gvec(\alphavecmicro_{0})$. The Gaussian \ac{pdf} of a random vector $\xvec \sim \mathcal{N}(\muvec, \Sigmamat)$ is denoted by $p(\xvec) = \mathcal{N}(\xvec, \muvec, \Sigmamat)$.

\vspace{-0.25cm}
\section{problem formulation and background}
\label{model_and_prob_section}
    In this section, we formulate the problem of non-Bayesian estimation under model misspecification.
    In Subsection \ref{model_subsec}, the problem formulation is presented, and in Subsection \ref{background_subsec}, the \ac{mcrb} is introduced. Finally, in Subsection \ref{issues_subsection} we present the limitations of the current version of the \ac{mcrb} framework and the crucial need for a new bound derivation.
    \vspace{-0.25cm}
    \subsection{Problem Formulation}
    \label{model_subsec}
        Non-Bayesian estimation under misspecified setting can be formulated as follows: Let $(\Omega_{\xvec},\mathcal{F}, P_{\varthetavecsmall})$ denote a probability space, where $\Omega_{\xvec}$ is the observation space, $\mathcal{F}$ is the $\sigma$-algebra, and $P_{\varthetavecsmall}$ is the true probability measure on $\mathcal{F}$, which is parameterized by a real deterministic parameter vector $\varthetavec \in \Omega_{\varthetavecsmall} \subseteq \mathbb{R}^{N_{1}}$, in which $\Omega_{\varthetavecsmall}$ is the parameter space. We assume that the \ac{pdf} \ac{wrt} $P_{\varthetavecsmall}$ exists and is denoted by $p(\xvec;\varthetavecmicro)$. We denote the true unknown deterministic parameter by $\varthetavec_{0}$ in the interior of $\Omega_{\varthetavecsmall}$.
        In addition, we consider the random observation vector, $\xvec \in \Omega_{\xvec}$, which is distributed according to $p(\xvec;\varthetavecmicro)$.
        Under the misspecified setting, $\xvec$ is assumed to be distributed according to the \ac{pdf} $f(\xvec;\thetavecmicro)$,  which is parameterized by a deterministic parameter vector $\thetavec \in \Omega_{\thetavecsmall} \subseteq \mathbb{R}^{N_{2}}$. 
        It should be noted that $N_2$ is not necessarily equal to $N_1$.
        Thus, $p(\xvec;\varthetavecmicro)$ is considered to be the {\em{true \ac{pdf}}}, while $f(\xvec;\thetavecmicro)$ is referred to as the {\em{assumed \ac{pdf}}}, considered to be strictly positive for any $\xvec \in \Omega_{\xvec}$ and $ \thetavec \in \Omega_{\thetavecsmall}$. 
        In the following, $\mathbb{E}_{p; \varthetavecsmall}[h(\xvec)]$ denotes the expectation \ac{wrt} the true \ac{pdf}, $p(\xvec; \varthetavecmicro)$, for any measurable function $h(\cdot)$, i.e. 
        $\mathbb{E}_{p; \varthetavecsmall}[h(\xvec)] = \int_{\Omega_{\xvec}} h(\alphavecmicro) p(\alphavecmicro; \varthetavecmicro) \ud \alphavecmicro$. Additionally, $\mathbb{E}_{p}[h(\xvec)]$ denotes the expectation \ac{wrt} the true \ac{pdf}, $p(\xvec; \varthetavecmicro_{0})$, for any measurable function $h(\cdot)$, i.e. $\mathbb{E}_{p}[h(\xvec)] = \mathbb{E}_{p; \varthetavecsmall_{0}}[h(\xvec)]$.
    \vspace{-0.25cm}
    \subsection{Background: MCRB}
    \label{background_subsec}
        In this subsection, we present the foundational background and key definitions behind the development of the \ac{mcrb}. 
        In estimation under the model misspecification setting, the parameter $\thetavec$ associated with the assumed \ac{pdf} does not necessarily coincide with the true parameter $\varthetavec$ governing the true \ac{pdf}. As an alternative, the pseudo-true parameter, which is the point that minimizes the \ac{kld} between the true and the assumed \acp{pdf}, is commonly used (see, e.g.  \cite{Fortunati_perf_bound,vuong1986cramer}).
        \begin{definition}[Pseudo-true parameter]
        \label{def:psuedo-true}
            For a true \ac{pdf} $p(\xvec;\varthetavecmicro)$ and an assumed \ac{pdf} $f(\xvec;\thetavecmicro)$, the pseudo-true parameter is defined as
            \beqna
            \label{eq:pseudo_def}
                \thetavec_*(\varthetavecmicro) = \arg\min_{\thetavecsmall \in \Omega_\thetavectiny} \mathcal{D}_{KL} \big(p(\cdot ;\varthetavecmicro) \| f(\cdot;\thetavecmicro) \big) ,
            \eeqna
            where $\mathcal{D}_{KL} (\cdot \| \cdot )$ is the \ac{kld} \cite[p. 211-212]{kay},
            and $f(\xvec;\thetavecmicro)$ is assumed to be strictly positive on the support of $p(\xvec;\varthetavecmicro)$,  $\forall \thetavec \in \Omega_{\thetavecsmall}$.
        \end{definition}
        
        Since $\mathbb{E}_{p; \varthetavecsmall} \big[ \log p(\xvec;\varthetavecmicro) \big]$ is not a function of $\thetavec$, the expression in \eqref{eq:pseudo_def} can be rewritten as \cite{harel2023non}
        \beqna
        \label{eq:pseudo_def_gen}
            \thetavec_*(\varthetavecmicro) \hspace{-0.25cm} 
            &=& \hspace{-0.25cm} \arg\max_{\thetavecsmall \in \Omega_\thetavectiny} \mathbb{E}_{p; \varthetavecsmall} \big[\log f(\xvec;\thetavecmicro) \big] .
        \eeqna
        It should be noted that Definition \ref{def:psuedo-true} is well-defined only if the minimization problem in \eqref{eq:pseudo_def} has a unique solution. This uniqueness is critical for the identifiability of the estimation problem, as discussed in \cite{RAN_identifiability}. In particular, the pseudo-true parameter in \eqref{eq:pseudo_def_gen} for $\varthetavec = \varthetavec_{0}$ is given by $\thetavec_{*}^{0} \define \thetavec_*(\varthetavecmicro_{0})$.
       
        The \ac{mcrb} relies on the following regularity conditions \cite{vuong1986cramer}:
        {\renewcommand{\theenumi}{C.\arabic{enumi}} 
        \begin{enumerate}
            \item
            \label{cond1}
                The pseudo-true parameter, defined in \eqref{eq:pseudo_def}, is unique.                     
            \item
            \label{cond2}
            The log-likelihoods, $\log f(\xvec;\thetavecmicro)$ and $\log p(\xvec;\varthetavecmicro)$, are twice differentiable \ac{wrt} $\thetavecmicro$ and $\varthetavecmicro$, respectively. 
                Furthermore, the expectations $ \mathbb{E}_{p;\varthetavecsmall} \big[\nabla_{\thetavecsmall} \log f(\xvec;\thetavecmicro) \big] $ and $ \mathbb{E}_{p;\varthetavecsmall} \big [{\nabla_{\thetavecsmall}}^{\!\!\!\! 2} \log f(\xvec;\thetavecmicro)\big]$ exist.
            
            \item
            \label{cond3} 
                The misspecified Hessian-based information matrix $\Amat(\thetavecmicro_{*}^{0})$ is non-singular, where
                \beqna
                \label{eq:def_A_theta}
                    \Amat(\thetavecmicro) \define -{\mathbb{E}}_{p} \big[ \nabla_{\thetavecsmall}^{2}  \log f(\xvec;\thetavecmicro) \big].
                \eeqna   
            \item
            \label{cond4}
                The operations of integration \ac{wrt} $\xvec$ and differentiation \ac{wrt} $\thetavec$ (and $\varthetavec$) of any measurable function $\gvec(\xvec;\thetavecmicro)$ (and $\hvec(\xvec;\varthetavecmicro)$) can be interchanged as follows: \begin{subequations}
                    \beqna
                    \label{eq:cyc_unbias_cond_1}
                        \hspace{-3.5cm} \nabla_{\thetavecsmall} \!\! \int_{\Omega_{\xvec}} \!\!\! \gvec(\alphavecmicro; \thetavecmicro) p(\alphavecmicro;\varthetavecmicro) \ud \alphavecmicro \hspace{-0.25cm} &=& \hspace{-0.35cm} \int_{\Omega_{\xvec}} \!\!\! \nabla_{\thetavecsmall} \gvec(\alphavecmicro; \thetavecmicro) p(\alphavecmicro;\varthetavecmicro) \ud \alphavecmicro,
                    \\ \hspace{-1.2cm}
                        \nabla_{\varthetavecsmall} \!\! \int_{\Omega_{\xvec}} \!\!\! \hvec(\alphavecmicro; \varthetavecmicro) p(\alphavecmicro;\varthetavecmicro) \ud \alphavecmicro \hspace{-0.25cm} &=& \hspace{-0.35cm} \int_{\Omega_{\xvec}} \!\!\! \nabla_{\varthetavecsmall} \Big\{ \! \hvec(\alphavecmicro; \varthetavecmicro) p(\alphavecmicro;\varthetavecmicro)\! \Big \} \! \ud \alphavecmicro.
                    \eeqna
                \end{subequations}    
        \end{enumerate}}
        In addition, the unbiasedness in a misspecified setting is defined as follows.
        \begin{definition}[Misspecified unbiasedness \cite{Fortunati_perf_bound}]
        \label{def:miss_unbias}
            Let the assumed \ac{pdf} $f(\xvec;\thetavecmicro)$ satisfy Condition \ref{cond1}. An estimator $\hat{\thetavec}(\xvec):\Omega_{\xvec} \rightarrow{\mathbb{R}^{N_{2}}}$ is a misspecified-unbiased estimator of $\thetavec_{*}^{0}$ \ac{wrt} $p(\xvec;\varthetavecmicro_{0})$ if
            \beqna
            \label{uniform_def_type1}
                \mathbb{E}_{p} \big[ \hat{\thetavec}(\xvec) -\thetavec_{*}^{0}\big] = \zerovec.
            \eeqna
        \end{definition}

        \begin{definition}[MCRB \cite{vuong1986cramer}]
        \label{mcrb_theorem}
            Let the assumed \ac{pdf} $f(\xvec;\thetavecmicro)$ satisfy Conditions \ref{cond1}-\ref{cond4}. 
            Then, the \ac{mcrb} is given by
            \beqna
            \label{eq:def_MCRB}
                {\text{\textbf{MCRB}}}(\thetavecmicro_{*}^{0}) \define \Amat^{-1}(\thetavecmicro_{*}^{0}) \Bmat(\thetavecmicro_{*}^{0}) \Amat^{-1}(\thetavecmicro_{*}^{0}),
            \eeqna
            where $\Amat(\thetavecmicro) $ is defined in \eqref{eq:def_A_theta}, and
            \beqna
            \label{eq:def_B_theta}
                \Bmat(\thetavecmicro) \define \mathbb{E}_{p} \Big[{\nabla_{\thetavecsmall}}^{\!\!\! T} \log f(\xvec;\thetavecmicro) \,\, \nabla_{\thetavecsmall} \log f(\xvec;\thetavecmicro) \Big]
            \eeqna
            is the gradient-based misspecified information term.
        \end{definition}
        
        The expectation operator for computing the different terms in Definition \ref{mcrb_theorem} is taken \ac{wrt} the true probability measure. It can be seen that in the conventional, well-specified case where $f(\xvec;\thetavecmicro^{0}_{*}) = p(\xvec;\varthetavecmicro_{0})$, the two information matrices are identical,  i.e. $ \Bmat(\thetavecmicro^{0}_{*}) = \Amat(\thetavecmicro^{0}_{*})$, and the r.h.s. of \eqref{eq:def_MCRB} is the \ac{crb}.

        The \ac{mml} estimator is defined as the \ac{ml} estimator of the unknown parameter under the assumed model \cite{white1982maximum}:
        \beqna
        \label{eq:def_MML}
            \hat{\thetavec}_{\text{MML}}(\xvec) \define \arg\max_{\thetavecsmall \in \Omega_\thetavectiny} f(\xvec; \thetavecmicro).
                 \vspace{-0.2cm}
        \eeqna
        Under mild regularity conditions, the \ac{mml} estimator converges asymptotically to the pseudo-true parameter, and its asymptotic misspecified covariance equals the \ac{mcrb} \cite{white1982maximum, huber1967behavior, Akaike1998}.

        \vspace{-0.2cm}
    \subsection{Limitations of the MCRB}
    \label{issues_subsection}
    \vspace{-0.05cm}
        The \ac{mcrb} has been widely investigated as the state-of-the-art tool for analyzing estimation performance under model misspecification \cite{Fortunati_Scatter_Matrix, Fortunati_perf_bound}. In many papers, it is a common belief that the following conjecture holds.
        \begin{conjecture}
        \label{instead_of_theorem}
            Let the assumed \ac{pdf} $f(\xvec;\thetavecmicro)$ satisfy Conditions \ref{cond1}-\ref{cond4} and $\hat{\thetavec}(\xvec)$ be any (locally) misspecified unbiased estimator as defined in Definition \ref{def:miss_unbias}. 
            Then,
           \beqna
            \label{eq:MSE_inq_MCRB}
                {\text{\textbf{MSE}}} \big(\hat{\thetavecmicro}(\xvec),\thetavecmicro_{*}^{0} \big) \succeq \text{\textbf{MCRB}}(\thetavecmicro_{*}^{0}),
            \eeqna
            where the \ac{mcrb} is defined in \eqref{eq:def_MCRB}, and 
            \beqna
            \label{eq:miss_MSE}
                {\text{\textbf{MSE}}} \big(\hat{\thetavecmicro}(\xvec),\thetavecmicro_{*}^{0} \big) \define \mathbb{E}_{p} [ (\hat{\thetavec}(\xvec) - \thetavec_{*}^{0} )(\hat{\thetavec}(\xvec) - \thetavec_{*}^{0} )^{T} ]
            \eeqna
            is the \ac{mse} of an estimator $\hat{\thetavec}(\xvec)$ evaluated at $\thetavec_{*}^{0}$.
        \end{conjecture}
        
        In the following, we show that this commonly-used conjecture is inaccurate. In particular, we demonstrate through a counterexample that the \ac{mcrb} may fail to bound the \ac{mse} of a misspecified unbiased estimator, revealing fundamental gaps in its current theoretical formulation and its practical application in signal processing. Specifically, in Subsection \ref{class_subsec}, we address the need to clearly state the conditions and specify the class of estimators for which the \ac{mcrb} provides a lower bound. In Subsection \ref{derivation_subsec}, we discuss the importance of providing a more direct, unified derivation of the bound. 
        
        A revised \ac{mcrb} theorem that resolves these issues is presented later (Theorem \ref{thm:Our_MCRB} in Section \ref{sub:Final_MCRB}), based on the developments in Sections \ref{naive_MCRB_section} and \ref{sec:equivalent_models}. 
        This theorem explicitly characterizes the class of estimators for which the \ac{mcrb} is valid, which differs from the class implied by Conjecture \ref{instead_of_theorem}.
         \begin{figure}
                \refstepcounter{myboxcounter} 
                \begin{mybox}[label=box:counter_example]{Counterexample for Conjecture \ref{instead_of_theorem}}
                    Consider the estimation problem where the true and assumed \acp{pdf} are
                    \begin{boxsubequations}
                        \beqna
                            p(\xvec; \vartheta) = \mathcal{N}(\xvec, \vartheta \onevec, \Sigmamat) , \hspace{0.15cm}\label{true_ex1}
                            \\
                            f(\xvec; \theta) = \mathcal{N}(\xvec, \theta \onevec, \sigma^2 \Imat_N), \label{assumed_ex1}
                        \eeqna
                    \end{boxsubequations}   
                    
                    \noindent where $\Sigmamat = \text{diag}\{\varepsilon, \sigma^{2}, \dots, \sigma^{2}\}$, and $\sigma^{2},\varepsilon > 0$.
                    The unknown parameters under the true and assumed \acp{pdf} are $\vartheta$ and $\theta$, respectively. By substituting this model in \eqref{eq:pseudo_def_gen}, it follows that in this case, $\theta_{*}(\vartheta) = \vartheta$ although the associated \acp{pdf} are different. By substituting this model in \eqref{eq:def_A_theta} and \eqref{eq:def_B_theta}, we obtain that the Hessian-and gradient-based information terms for this case are given by 
                    \begin{boxsubequations}
                        \beqna
                            \label{A_sc}
                            A(\theta_{*}) \hspace{-0.25cm} &=& \hspace{-0.25cm} \frac{N}{\sigma^{2}}, 
                            \\
                            \label{B_sc}
                            B(\theta_{*}) \hspace{-0.25cm} &=& \hspace{-0.25cm} \frac{1}{\sigma^{4}} \onevec^{T} \Sigmamat \onevec = \frac{\varepsilon + (N-1) \sigma^{2}}{\sigma^{4}}.
                        \eeqna
                    \end{boxsubequations}
                    
                    \noindent Then, substitution of \eqref{A_sc} and \eqref{B_sc} into \eqref{eq:def_MCRB}, results in the following \ac{mcrb} 
                    \begin{boxequation}
                    \label{mcrb_ex1}
                        \text{MCRB}(\theta_{*}^{0}) = \frac{1}{N^{2}} \onevec^{T} \Sigmamat \onevec = \frac{\varepsilon + (N-1) \sigma^{2}}{N^{2}}.
                    \end{boxequation}
                    Now, consider the estimator based only on the first entry of $\xvec$, $\hat{\theta}_{1}(\xvec) = x_{1}$. This estimator is a misspecified unbiased estimator of $\theta_{*}^{0}$ according to Definition \ref{def:miss_unbias}, since by using \eqref{true_ex1}, we have $\mathbb{E}_{p}[x_{1}] = \vartheta_{0} = \theta_{*}^{0}$. In addition, the \ac{mse} of this estimator is
                    \begin{boxequation}
                    \label{mse_ex1}
                        {\text{MSE}} (\hat{\theta}_{1}(\xvec),\theta_{*}^{0} ) = \mathbb{E}_{p} [ \big(x_1 - \vartheta_{0} \big)^2 ] = \varepsilon.
                    \end{boxequation}
                    
                    For $\varepsilon < \frac{\sigma^{2}}{N + 1}$, the \ac{mse} in \eqref{mse_ex1} is lower than the \ac{mcrb} in \eqref{mcrb_ex1}. Thus, the \ac{mcrb} from Definition \ref{mcrb_theorem} does not provide a valid lower bound on the class of misspecified unbiased estimators in this example, and the inequality in \eqref{eq:MSE_inq_MCRB} may not hold. 
                    
                    \hspace{0.05em} The \ac{mml} estimator for this case is $\hat{\theta}_{\text{MML}}(\xvec) = \frac{1}{N} \onevec^{T} \xvec$, which is also a misspecified unbiased estimator of $\theta_{*}^{0}$ with  ${\text{MSE}} (\hat{\theta}_{\text{MML}}(\xvec),\theta_{*}^{0}) = \text{MCRB}(\theta_{*}^{0})$ (see \eqref{mcrb_ex1}). 
                    
                    \hspace{0.05em} In Fig. \ref{fig:misspecified_gaussian_example}, we present the empirical \ac{mse} of $\hat{\theta}_{1}(\xvec)$ and $\hat{\theta}_{\text{MML}}(\xvec)$, the oracle CRB of the true model, and the MCRB from \eqref{mcrb_ex1} versus the number of samples $N$. It can be seen that the MCRB fails to bound the MSE of $\hat{\theta}_{1}(\xvec)$, while the oracle \ac{crb} is not tight and is irrelevant for this problem.
                    
                    \vspace{0.15cm}
                    \centering
                    \includegraphics[width=0.85\linewidth]{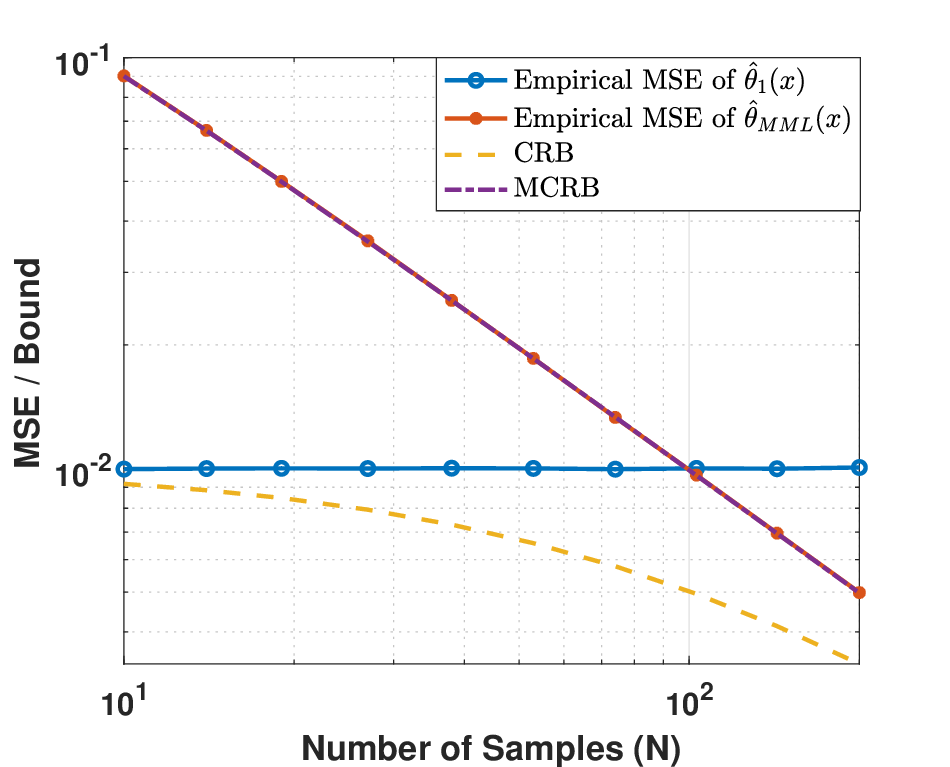}
                    \caption{MSE and bounds for the above example. }
                    \label{fig:misspecified_gaussian_example}

                \end{mybox} 
            \end{figure}    
        \subsubsection{Class of estimators}
        \label{class_subsec}
            In Conjecture \ref{instead_of_theorem}, the only requirement imposed on the estimator is that $\hat{\thetavec}(\xvec)$ satisfy the misspecified unbiasedness condition defined in Definition \ref{def:miss_unbias}. However, recent works show that the \ac{mcrb} does not always provide a valid lower bound on the \ac{mse} of practical misspecified unbiased estimators \cite{rosentha2024asymptotically}. 
            A simple example in which a misspecified unbiased estimator has a lower \ac{mse} is presented in Box \ref{box:counter_example}.
            This counterexample arises from the fact that the class of estimators whose \ac{mse} is bounded by the \ac{mcrb} has not been clearly characterized. 
            As a result, the \ac{mcrb} cannot, in general, be  regarded as a valid lower bound on the \ac{mse} without explicitly specifying the class of estimators or models to which it effectively applies. This issue is further clarified in Subsection \ref{sub:Gen_Models}, where we mathematically show that, in the general case, the bound derived under the unbiasedness condition alone is less tight than the \ac{mcrb}.

            Identification of the relevant class of estimators is essential both for the theoretical rigor of the \ac{mcrb} and for practical guidance on designing estimators
            under model misspecification. 
            Furthermore, such identification could support the development of alternative bounds, e.g. the Barankin bound, tailored to different estimator classes.
           
        \subsubsection{Methodological Derivation}
        \label{derivation_subsec}
            The \ac{mcrb} has been obtained through multiple different derivations, primarily in \cite{vuong1986cramer} and \cite{Richmond_Howorits_jour}. The first derivation of the \ac{mcrb} as a lower bound on the \ac{mse} was provided in \cite{vuong1986cramer} and later formulated in signal-processing terms in \cite{Richmond_Fusion_2018}. This derivation was initially established for the case of a parametric true model. An extension for the case of the non-parametric true model was performed by employing a specific choice of a distribution class to obtain the \ac{mcrb} in Definition \ref{mcrb_theorem}. 
            Despite its importance, this derivation suffers from several shortcomings.
            First, it relies on the introduction of a distribution family (``least favorable" distribution),  where it is unclear how to choose this family, and it is not standard in the signal processing community. This reliance complicates interpretation and direct comparison with other bounds.
            Second, the limitations of the bound derived for the parametric case are not explicitly discussed, leaving unclear why further refinement of the bound is necessary. Finally, the class of estimators to which the \ac{mcrb} applies (see Subsection \ref{class_subsec}) is not specified, and only misspecified unbiasedness, as defined in Definition \ref{def:miss_unbias}, is mentioned as a requirement which is insufficient, as discussed above.
            
            In \cite{Richmond_Howorits_jour}, an alternative derivation of the \ac{mcrb} was proposed using the covariance inequality and an auxiliary function, which fits the classical approach of developing covariance inequality-based performance bounds in signal processing and estimation theory. However, the final step in the derivation is based on a first-order approximation around the \ac{mml} estimator, utilizing its relation with the \ac{mcrb} from \cite{white1982maximum}.
            Despite these significant contributions, there remains a gap for a clear and rigorous derivation of the \ac{mcrb} that is not based on considering a specific estimator (here, \ac{mml} estimator), and without asymptotic approximations. 
  \vspace{-0.2cm}
\section{Naive \ac{mcrb} under Local Misspecified Unbiasedness}
\label{naive_MCRB_section}
  \vspace{-0.05cm}
    As described in Section \ref{model_and_prob_section}, there is a need for a rigorous derivation of the \ac{mcrb} that clearly states the conditions under which the bound holds and specifies the relevant class of estimators.
    In this section, we develop an \ac{mse} lower bound based solely on local misspecified unbiasedness. 
    To this end, in Subsection \ref{sub:Unbiasedness}, we present the local unbiasedness under model misspecification. Then, in Subsection \ref{sub:Gen_Models}, we develop the corresponding \ac{mse} bound as the solution to a minimization problem. Finally, in Subsection \ref{limitation_Th2_subsection}, we demonstrate that, in the general case, this development may lead to a non-informative bound. 
    Nevertheless, understanding why this naive approach of using only local misspecified unbiasedness is inadequate, provides essential background that motivates the derivation of the \ac{mcrb} in Section \ref{sec:equivalent_models}.
    \vspace{-0.25cm}
    \subsection{Misspecified Unbiasedness}
    \label{sub:Unbiasedness}
        In order to derive a Cram$\acute{\text{e}}$r-Rao type lower bound on the \ac{mse}, an appropriate notion of local unbiasedness should be defined. 
        Similarly to local mean-unbiasedness (see e.g. \cite{menni2011new}), the misspecified unbiasedness can be defined locally by requiring that \eqref{uniform_def_type1} will be satisfied in the vicinity of the true parameter $\varthetavec_0$, as defined in the following definition.
        \begin{definition}[Naive local misspecified unbiasedness]
        \label{local_m_unbiasedness1}
            Let the assumed \ac{pdf} $f(\xvec;\thetavecmicro)$ satisfy Condition \ref{cond1}. An estimator $\hat{\thetavec}(\xvec)$ is a naive {\em{locally}} misspecified-unbiased estimator of $\thetavec_{*}^{0}$ \ac{wrt} $p(\xvec;\varthetavecmicro_{0})$ if at $\varthetavec_0\in\Omega_\varthetavecsmall$ it satisfies
            \begin{subequations}
                \beqna
                    {\mathbb{E}}_{p} \big[\hat{\thetavec}(\xvec) - \thetavec_{*}^{0} \big] \hspace{1.4cm} &=& \zerovec , \label{pointwise_m_unbiasedness_1}
                \\
                    \nabla_{\varthetavecsmall} \mathbb{E}_{p; \varthetavecsmall} \big[\hat{\thetavec}(\xvec) - \thetavec_{*}(\varthetavec) \big]\big|_{\varthetavecsmall=\varthetavecsmall_0} &=& \zerovec . \label{local_m_unbiasedness_1}
                \eeqna
            \end{subequations} 
        \end{definition}
        
        \begin{claim}
        \label{claim_der_pt}
            Let the assumed \ac{pdf} $f(\xvec;\thetavecmicro)$ satisfy Conditions \ref{cond1}-\ref{cond4}. The requirement in \eqref{local_m_unbiasedness_1} can be reformulated as
            \beqna
            \label{17_tmp}
                {\mathbb{E}}_{p} \big[ \big(\hat{\thetavec}(\xvec)-\thetavec_{*}^{0}\big) \, \nabla_{\varthetavecsmall} \log p(\xvec; \varthetavecmicro_{0}) \big] \hspace{2.5cm}\nonumber\\= {\Amat}^{-1} ( \thetavecmicro_{*}^{0} ) \Bmat_{p;f}( \thetavecmicro_{*}^{0},\varthetavecmicro_{0} ),
            \eeqna  
            where $\Amat(\thetavecmicro) $ is defined in \eqref{eq:def_A_theta} and
            \beqna
            \label{eq:def_B_p_f}
                \Bmat_{p;f}(\thetavecmicro,\varthetavecmicro) \define {\mathbb{E}}_{p} \big[ {\nabla_{\thetavecsmall}}^{\!\!\! T} \log f(\xvec;\thetavecmicro) \,\, \nabla_{\varthetavecsmall} \log p(\xvec;\varthetavecmicro) \big].
            \eeqna
        \end{claim}
        
        \begin{IEEEproof}
            The proof appears in Appendix \ref{App:der_PT}.
        \end{IEEEproof}

        {\em{Discussion:}} 
            Definition \ref{local_m_unbiasedness1} requires $\hat{\thetavec}(\xvec)$ to be a mean-unbiased estimator of the pseudo-true parameter, $\thetavec_{*}^{0}$, and that its first-order derivative of the estimation error w.r.t.\ $\varthetavec$ vanishes at $\varthetavec_0$. However, these local unbiasedness conditions rely on the true log-likelihood function, $\log p(\xvec;\varthetavecmicro)$. Moreover, the local unbiasedness condition in \eqref{local_m_unbiasedness_1} requires the estimator's bias derivative to vanish \ac{wrt} perturbations in $\varthetavec$, which implicitly assumes knowledge of the true \ac{pdf} $p(\xvec;\varthetavecmicro)$. This is an unrealistic assumption when the model is misspecified.
            Hence, these conditions do not fundamentally incorporate the misspecification aspect, beyond the use of the pseudo-true parameter. 
            This can also be seen from the local unbiasedness condition in \eqref{17_tmp}, which imposes a constraint on the correlation between the partial derivative of the true log-likelihood and the estimation error. 

            On another note, using the estimator bias \ac{wrt} the true parameter $\varthetavec_{0}$,  conditions \eqref{pointwise_m_unbiasedness_1}-\eqref{local_m_unbiasedness_1} can be written as
            \begin{subequations}
                \beqna
                    \mathbb{E}_{p} \big[\hat{\thetavec}(\xvec)-\varthetavec_{0} \big] \hspace{0.55cm} &=& \thetavec_{*}^{0} - \varthetavec_{0} , \label{pointwise_m_unbiasedness_1_b}
                \\
                    \hspace{-0.65cm} \nabla_{\varthetavecsmall} \mathbb{E}_{p; \varthetavecsmall} \big[\hat{\thetavec}(\xvec) - \varthetavec \big]\big|_{\varthetavecsmall=\varthetavecsmall_0} \hspace{-0.3cm} &=&  \nabla_{\varthetavecsmall} \big\{ \thetavec_{*}(\varthetavec)-\varthetavec \big\} \big|_{\varthetavecsmall=\varthetavecsmall_0} . \label{local_m_unbiasedness_1_b}
                \eeqna
            \end{subequations}  
            These conditions can be interpreted as simply enforcing a specific bias on the estimator of the true unknown parameter. Thus, it aligns with a setting of conventional biased estimation \cite{KAY_ELDAR}. Hence, the model misspecification aspect is not fully integrated into the family of estimators. 
            

        \vspace{-0.25cm}
    \subsection{Lower Bound under Locally Misspecified Unbiasedness}
    \label{sub:Gen_Models}
        { \samepage In this subsection,  we derive
        a lower bound on the \ac{mse} under a naive locally misspecified unbiasedness restriction, as defined in Definition \ref{local_m_unbiasedness1}. We use the notation 
        \beqna
        \label{J_def}
            \Jmat_{p}(\varthetavecmicro)\define \mathbb{E}_{p;\varthetavecsmall} \big[ {\nabla_{\varthetavecsmall}}^{\!\!\! T} \log p(\xvec;\varthetavecmicro) \,\, \nabla_{\varthetavecsmall} \log p(\xvec;\varthetavecmicro) \big]
        \eeqna
        for the true-model \ac{fim}.}
        \begin{theorem}[Naive MCRB]
        \label{mcrb_theorem_5_terms}
            Let the assumed \ac{pdf} $f(\xvec;\thetavecmicro)$ satisfy Conditions \ref{cond1}-\ref{cond4}, and $\hat{\thetavec}(\xvec)$ be a naive locally misspecified unbiased estimator, defined in Definition \ref{local_m_unbiasedness1}. Then,
            \beqna
            \label{eq:MSE_inq_MCRB2}
                {\text{\textbf{MSE}}} \big(\hat{\thetavecmicro}(\xvec),\thetavecmicro_{*}^{0} \big) \succeq {\text{\textbf{nMCRB}}}_{p}(\thetavecmicro_{*}^{0}, \varthetavecmicro_{0}),
            \eeqna
            where 
            \beqna
            \label{nMCRB_def}
                {\text{\textbf{nMCRB}}}_{p}(\thetavecmicro_{*}^{0},\varthetavecmicro_{0}) \hspace{5.5cm} \nonumber \\ \define {\Amat}^{-1}(\thetavecmicro_{*}^{0}) \Bmat_{p;f}(\thetavecmicro_{*}^{0}, \varthetavecmicro_{0}) \Jmat_{p}^{-1}(\varthetavecmicro_{0}) \Bmat_{p;f}^{T}(\thetavecmicro_{*}^{0},\varthetavecmicro_{0}) {\Amat}^{-1}(\thetavecmicro_{*}^{0}),
            \eeqna
            and $\Amat(\cdot)$, $\Bmat_{p;f}(\cdot,\cdot)$, and $\Jmat_{p}(\cdot)$ are defined in \eqref{eq:def_A_theta}, \eqref{eq:def_B_p_f}, and \eqref{J_def}, respectively.
            Equality is attained in \eqref{eq:MSE_inq_MCRB2} {\em{iff}}
            \begin{align}
            \label{eq:MVU_bCRB}
                \hat{\thetavec}(\xvec) = \thetavec_{*}^{0}  + \hspace{-0.05cm} {\Amat}^{-1}(\thetavecmicro_{*}^{0}) \Bmat_{p;f}(\thetavecmicro_{*}^{0},\varthetavecmicro_{0}) \Jmat_{p}^{-1}(\varthetavecmicro_{0}) \, {\nabla_{\varthetavecsmall}}^{\!\!\! T} \! \log p(\xvec;\varthetavecmicro_{0}).
            \end{align}
        \end{theorem}
        
        \begin{IEEEproof}
            The proof appears in Appendix \ref{App:naive_mcrb_theorem}.
        \end{IEEEproof}
        
        It should be noted that the naive \ac{mcrb} in \eqref{nMCRB_def} is developed in Theorem 3.1 in \cite{vuong1986cramer} based on Cauchy-Schwarz inequality, and serves as an intermediate step in the derivation of the \ac{mcrb}.
        This bound can also be derived from the biased \ac{crb} of the true model, as shown in \cite{Richmond_Fusion_2018}. That is, assuming the bias is given by \eqref{pointwise_m_unbiasedness_1_b}-\eqref{local_m_unbiasedness_1_b} and substituting the gradient of this bias in the biased \ac{crb} \cite[p. 37-39]{kay} results in the bound in \eqref{eq:b-CRB}.
        Our derivation via a minimization approach highlights that, under the sole requirement of local misspecified unbiasedness, the bound in \eqref{nMCRB_def} is the tightest 
        \ac{mse} lower bound.
    \vspace{-0.5cm}
    \subsection{Properties and Limitations of the Naive MCRB }
    \label{limitation_Th2_subsection}
        In this subsection, we investigate the properties of the naive \ac{mcrb} from Theorem \ref{mcrb_theorem_5_terms}. We show how it relates to the classical \ac{mcrb} of Definition \ref{mcrb_theorem} and to the oracle \ac{crb} associated with the true model. We also discuss its main drawback: although it provides a lower bound on the class of locally misspecified unbiased estimators, it can become non-informative in many realistic model misspecification scenarios. 
        Thus, this subsection underscores the need to define the relevant class of misspecified estimators and the appropriate unbiasedness to obtain a meaningful and attainable bound.
        \subsubsection{Order Relation}
            The following claim states that the \ac{mcrb} is a tighter lower bound than the naive \ac{mcrb}.
            \begin{claim}
            \label{cl:MCRB_vs_bCRB}
                The classical \ac{mcrb} in \eqref{eq:def_MCRB} is always tighter than or equal to the naive \ac{mcrb} in \eqref{nMCRB_def}.
            \end{claim}
            
            \begin{IEEEproof}
                The proof appears in Appendix \ref{App:CS-MCRB-bCRB}.
            \end{IEEEproof}
            
            Claim \ref{cl:MCRB_vs_bCRB} indicates that in the general case, the naive \ac{mcrb} offers a weaker guarantee on the estimation performance than the classical \ac{mcrb}, which more accurately captures the asymptotic behavior of misspecified estimators. In particular, the \ac{mcrb} characterizes the asymptotic performance of the \ac{mml} estimator \cite{white1982maximum}. This order relation can be explained by the fact that the naive approach leverages information that is not available to an estimator solely using $f(\xvec;\thetavecmicro)$, thus failing to reflect misspecification constraints and is too optimistic. 
 
        \subsubsection{Case of $\thetavec_{*} = \varthetavec$}
        \label{subsubsec:pseudo_true_equals_true}
            In this subsection, we describe the case of a misspecified setting where the pseudo-true parameter equals the true parameter, i.e. $\thetavec_{*} = \varthetavec$, even though the model is not correctly specified, i.e. $f(\xvec; \thetavecmicro_{*}) \neq p(\xvec; \varthetavecmicro)$. An example of this case is presented in Box \ref{box:counter_example_cont_1}, where we apply the naive \ac{mcrb} to the example in Box \ref{box:counter_example}.
            \begin{figure}
                \refstepcounter{myboxcounter} 
                \begin{mybox}[label=box:counter_example_cont_1]{Inappropriateness of the Naive MCRB}
                We revisit Box \ref{box:counter_example} to show that, although the naive \ac{mcrb} provides a valid lower bound, it is not an appropriate bound in this case.
                    Substituting the model from Box \ref{box:counter_example} in \eqref{eq:def_B_p_f} and \eqref{J_def}, we obtain
                    \vspace{-0.25cm}
                    \begin{boxsubequations}
                        \beqna
                            \label{J_box1}
                            J_{p}(\vartheta) \hspace{-0.1cm} &=& \hspace{-0.15cm} (N-1)/ \sigma^{2} + 1/\varepsilon,
                            \\
                            \label{Bpf_box1}
                            B_{p;f}(\theta_{*}, \vartheta) \hspace{-0.15cm} &=& \hspace{-0.15cm} N/\sigma^{2}.
                        \eeqna
                    \end{boxsubequations}
                    
                    \noindent By substituting \eqref{A_sc}, \eqref{J_box1}, and \eqref{Bpf_box1} in \eqref{nMCRB_def}, we obtain that the naive \ac{mcrb} is given by
                    \begin{boxequation}
                        {\text{nMCRB}}_{p}(\theta_{*}^{0},\vartheta_{0})= \frac{\varepsilon}{(N-1)\varepsilon/ \sigma^{2} + 1}.
                        \label{nMCRB_Box2}
                    \end{boxequation}
                    Now, consider the estimator $\hat{\theta}(\xvec) = x_{1}$. First, it is locally unbiased as defined in Definition \ref{local_m_unbiasedness1}, since
                    \begin{boxsubequations}
                        \beqna
                            {\mathbb{E}}_{p} [x_1 - \theta_{*}^{0}] \!\!\!\! &=& \!\!\! 0,
                            \\
                            \label{con2_unb_ex1}
                            \mathbb{E}_{p} [ (x_{1} - \theta_{*}^{0}) \nabla_{\vartheta}\! \log p(\xvec;\vartheta_{0})  ] \!\!\!\! \hspace{-0.05cm}&=&\hspace{-0.05cm} \!\!\! A^{-1}(\theta_{*}^{0}) B_{p;f}(\theta_{*}^{0}, \vartheta_{0})  =  1. \nonumber \\
                        \eeqna
                    \end{boxsubequations}
                    
                    \noindent Second, the \ac{mse} of the estimator $\hat{\theta}(\xvec) = x_1$ from \eqref{mse_ex1} is bounded by the naive \ac{mcrb} from \eqref{nMCRB_Box2}:
                    \begin{boxequation}
                        \text{MSE}(x_{1}, \theta_{*}^{0}) = \varepsilon \geq \frac{ \varepsilon }{(N-1)\varepsilon/ \sigma^{2} + 1} .
                    \end{boxequation}
                    
                    \noindent Thus, the naive \ac{mcrb} is a valid lower bound for the estimator $\hat{\theta}(\xvec) = x_{1}$, as expected.
    
                    \hspace{0.05em} In addition, the \ac{ml} estimator under the true model,
                    \begin{boxequation}
                    \label{ML_box2}
                        \hat{\vartheta}_{ML}(\xvec) = \frac{1}{\onevec^{T} \Sigmamat^{-1} \onevec} \onevec^{T} \Sigmamat^{-1} \xvec,
                    \end{boxequation}
                    
                    \noindent is also a naive misspecified unbiased estimator as defined in Definition \ref{local_m_unbiasedness1}.
                    Moreover, in this case the naive \ac{mcrb} coincides with the oracle \ac{crb} \ac{wrt} $p(\xvec; \vartheta_{0})$ and is attained by the ``oracle" \ac{ml} estimator.
                    However, this estimator uses the true \ac{pdf} $p(\xvec; \vartheta_{0})$  through $\Sigmamat$ and thus, cannot be considered as a valid estimator.
                    This exemplifies that the naive \ac{mcrb} and the naive misspecified unbiasedness condition characterize an oracle scenario rather than the misspecified problem of interest and fail to properly capture the effect of model misspecification.
                \end{mybox} 
            \end{figure}    
            In this case, the derivative of the pseudo-true parameter is
            \beqna
            \label{derivative_equal_1}
                \nabla_{\varthetavecsmall} \, \thetavec_{*}(\varthetavecmicro) = \Imat_{N_{1}}, ~~ \forall \varthetavec \in \Omega_{\varthetavecsmall}.
            \eeqna
            According to \eqref{eq:der_PT} from Appendix \ref{App:der_PT} (see also Equation (3.7) in \cite{vuong1986cramer}), \eqref{derivative_equal_1} implies that the term in \eqref{eq:def_B_p_f} coincides with the misspecified information matrix  in \eqref{eq:def_A_theta}, i.e. 
            \beqna
            \label{A_equal_B}
                \Bmat_{p;f} \big( \thetavecmicro_{*}(\varthetavecmicro),\varthetavecmicro \big) = \Amat \big( \thetavecmicro_{*}(\varthetavecmicro) \big), ~~ \forall \varthetavecmicro \in \Omega_{\varthetavecsmall}.
            \eeqna           
            Substituting \eqref{A_equal_B} into \eqref{nMCRB_def} shows that, in this case,  the naive bound from Theorem \ref{mcrb_theorem_5_terms} reduces to the conventional (oracle) \ac{crb} \ac{wrt} $p(\xvec;\varthetavecmicro_{0})$, i.e. 
            \beqna
            \label{eq:conventional_crb}
                {\text{\textbf{MSE}}} \big(\hat{\thetavecmicro}(\xvec),\varthetavecmicro_{*}^{0} \big) \succeq \Jmat_{p}^{-1}(\varthetavecmicro_{0}) = {\text{\textbf{CRB}}}(\varthetavecmicro_{0}),
            \eeqna
            where ${\text{\textbf{CRB}}}(\varthetavecmicro_{0})$ is the \ac{crb} using the true model.
            Since $\thetavec_{*} = \varthetavec$, the set of locally misspecified estimators defined in Definition \ref{local_m_unbiasedness1} is equal to the set of locally unbiased estimators in the classical sense. Therefore, the \ac{crb} is a valid lower bound.
            However, the \ac{crb} remains an ``oracle'' bound, which is not informative, as it does not account for model misspecification.
            As such, it characterizes the asymptotic performance of the \ac{ml} estimator under the true model, which is impractical, and therefore does not accurately characterize the asymptotic performance of the \ac{mml} estimator, which is the classical estimator used under model misspecification \cite{white1982maximum}.
        
        \vspace{-0.05cm}
        \subsubsection{Unattainability of the naive MCRB}  
            A key limitation of the naive \ac{mcrb} is that under a general model of misspecification,  realistic estimators should be based on the assumed \ac{pdf}, $f(\xvec;\thetavecmicro)$, lacking access to $p(\xvec;\varthetavecmicro)$. However, as can be seen in \eqref{eq:MVU_bCRB}, the estimator that achieves the naive \ac{mcrb} in \eqref{eq:MSE_inq_MCRB2} is inevitably a function of the true \ac{pdf}, $p(\xvec;\varthetavecmicro)$, which is unavailable in misspecified scenarios. 
            As a result, the lower bound in \eqref{nMCRB_def} 
            is too optimistic for general mismatch scenarios.

\vspace{-0.25cm}
\section{New Derivation of the MCRB}  
\label{sec:equivalent_models}
\vspace{-0.05cm}
    In Subsection \ref{limitation_Th2_subsection}, it was demonstrated that a misspecified bound based solely on local misspecified unbiasedness leads to an overly optimistic and generally unattainable bound, named here the naive \ac{mcrb}, since the unbiasedness condition characterizes a set of estimators that is unrealistic in a misspecified setting. While knowledge of the true \ac{pdf} $p(\xvec;\varthetavecmicro_{0})$ is necessary for performance analysis, the unbiasedness conditions placed on the estimator must not imply a dependence on the true \ac{pdf}. This balance will lead to a proper definition of local unbiasedness, enabling the derivation of a tighter and appropriate misspecified Cram$\acute{\text{e}}$r-Rao type bound.
    
    In this section, we introduce a new derivation of the \ac{mcrb} that aims to achieve this balance.
    The key idea is to embed the true \ac{pdf} $p(\xvec;\varthetavecmicro_{0})$ into a broader family of auxiliary \acp{pdf} that preserve only its pointwise properties.
    This approach limits the information that can be used by the bound and by the associated unbiasedness conditions.
    We derive the naive \ac{mcrb} from Subsection \ref{sub:Gen_Models} for the different auxiliary \acp{pdf}. By maximizing over this family, we identify the pointwise equivalent \ac{pdf} that yields the tightest \ac{mse} lower bound. We then show that the resulting bound coincides with the classical \ac{mcrb} from Definition \ref{mcrb_theorem}, thereby justifying its use.

    The steps for the proposed \ac{mcrb} derivation are as follows:
    \begin{enumerate}
        \item Introduce pointwise equivalent models by constructing an embedded family of \acp{pdf} that are pointwise equivalent to $p(\xvec;\varthetavecmicro_0)$ (Subsection \ref{sub:Model_trans}).
        
        \item Present the naive \ac{mcrb} for a generic pointwise equivalent model (Subsection \ref{mcrb_general_derivation}).
        
        \item Maximize the bound over the family from Step 1) to obtain the tightest bound, and show that it coincides with the \ac{mcrb} (Subsection \ref{sub:opt_equi}).

    \end{enumerate}

    \vspace{-0.25cm}
    \subsection{Pointwise Equivalent Models}
    \label{sub:Model_trans}
        In this subsection, we formalize the concept of pointwise equivalent models for a general estimation problem under misspecification, as described in Subsection \ref{model_subsec}. 
            \vspace{-0.1cm}
        \begin{definition}[Pointwise equivalent PDFs]
        \label{def:point_wise_equ}
            A \ac{pdf}, $\tilde{p}(\xvec; \gammavecmicro)$ with $\gammavec \in \Omega_{\gammavecsmall}\subseteq \mathbb{R}^{N_{3}}$ is a pointwise equivalent \ac{pdf} of the true \ac{pdf} $p(\xvec; \varthetavecmicro)$ at $\gammavec = \gammavec_{0}$ if
            \beqna
            \label{eq:point_wise_equ}
                \tilde{p}(\xvec; \gammavecmicro_{0}) = p(\xvec ;\varthetavecmicro_{0}),
            \eeqna
            where $\varthetavec_0$ is the given true parameter and the equality is in the almost-sure sense \ac{wrt} the true distribution.
        \end{definition}
               \vspace{-0.1cm}
        In addition, we assume that the pointwise equivalent \acp{pdf}, $\Tilde{p}(\xvec; \gammavecmicro)$, satisfy Conditions \ref{cond1}, \ref{cond2}, and \ref{cond4}, where $p(\xvec;\varthetavecmicro)$  is replaced with $ \tilde{p}(\xvec; \gammavecmicro)$ in those conditions. In particular, according to Condition \ref{cond2}, we require that the pointwise equivalent log-likelihood, $\log \tilde{p}(\xvec; \gammavecmicro)$, is twice differentiable \ac{wrt} $\gammavec$. However, it should be noted that the derivatives of the log-likelihood functions of the true \ac{pdf} and its pointwise equivalents, 
        $\nabla_{\varthetavecsmall} \log p(\xvec; \varthetavecmicro)$ and $\nabla_{\gammavecsmall} \log \tilde{p}(\xvec; \gammavecmicro)$
        may not be equal even at the specific values of $\varthetavecmicro = \varthetavecmicro_{0}$ and $\gammavec = \gammavec_{0}$.  Hence, $\tilde{p}(\xvec;\gammavecmicro)$ is only required to coincide with $p(\xvec;\varthetavecmicro)$ in a pointwise manner at one specific parameter pair, while it may generally differ at other parameter values.
        Now, a pointwise equivalent misspecified parameter estimation model is the model described in Subsection \ref{model_subsec}, where $p(\xvec;\varthetavecmicro)$ is replaced with a pointwise equivalent \ac{pdf}, $\tilde{p}(\xvec;\gammavecmicro)$. 
        According to \eqref{eq:pseudo_def_gen}, the pseudo-true parameter under this equivalent model, denoted by $\tilde{\thetavec}_{*}(\gammavecmicro)$, is  
        \beqna
        \label{eq:pseudo_def_gen_equivalent}
            \tilde{\thetavec}_{*}(\gammavecmicro) = \arg\max_{\thetavecsmall \in \Omega_\thetavectiny} \mathbb{E}_{\tilde{p}; \gammavecsmall} [ \log f(\xvec;\thetavecmicro)]; 
        \eeqna
        The following proposition states that the equivalent-model pseudo-true parameter, $\tilde{\thetavec}_{*}(\gammavecmicro)$, is equal to the original pseudo-true parameter computed with the true \ac{pdf}, $\thetavecmicro_{*}^{0}$, at $\gammavec = \gammavec_{0}$. 
        \begin{proposition}[Pseudo-true parameter invariance]
            Let $\tilde{p}(\xvec; \gammavecmicro)$ be a pointwise equivalent \ac{pdf} of the true \ac{pdf} according to Definition \ref{def:point_wise_equ}. 
            Then, the pseudo-true parameters under the two models are equal at $\gammavec = \gammavec_0$, i.e. 
            \beqna 
            \label{pseudo1_pseudo2}
                \thetavec_{*}^{0} = \tilde{\thetavec}_{*}(\gammavecmicro_0).
            \eeqna
        \end{proposition}
        
        \begin{IEEEproof}
            Evaluating \eqref{eq:pseudo_def_gen_equivalent} at $\gammavec=\gammavec_0$, we obtain 
            \beqna
            \label{eq:pseudo_tilde_def}
                \tilde{\thetavec}_{*}(\gammavecmicro_0) 
                = 
                \arg\max_{\thetavecsmall \in \Omega_{\thetavectiny}} \int_{\Omega_{\xvec}} \log f(\alphavecmicro;\thetavecmicro) \, p(\alphavecmicro;\varthetavecmicro_{0}) \ud \alphavecmicro = \thetavec_{*}^{0},
            \eeqna
            where the first and second equalities follow from the pointwise equivalence in \eqref{eq:point_wise_equ} and the definition of the pseudo-true parameter in \eqref{eq:pseudo_def_gen} for $\varthetavec = \varthetavec_{0}$, respectively.
        \end{IEEEproof}

        The expectation of any measurable function $h(\xvec)$ \ac{wrt} the true \ac{pdf}, $p(\xvec; \varthetavecmicro_{0})$, is the same as the expectation of $h(\xvec)$ \ac{wrt} any pointwise equivalent \ac{pdf} at $\gammavec = \gammavec_0$, $\tilde{p}(\xvec; \gammavecmicro_{0})$, which is denoted by ${\mathbb{E}}_{\tilde{p}}[h(\xvec)] = \mathbb{E}_{\tilde{p}; \gammavecsmall_{0}}[h(\xvec)]$.
        In particular, the \ac{mse} matrix of a misspecified estimator using $p(\xvec; \varthetavecmicro_{0})$ is identical to that obtained under $\tilde{p}(\xvec; \gammavecmicro_{0})$, i.e. 
        \beqna
        \label{MSE_equivalent}
            \mathbb{E}_{\tilde{p}} \big[ (\hat{\thetavecmicro}(\xvec) -  \tilde{\thetavecmicro}_{*}(\gammavecmicro_0) ) (\hat{\thetavecmicro}(\xvec) -  \tilde{\thetavecmicro}_{*}(\gammavecmicro_0) )^{T} \big]  \hspace{2cm} \nonumber \\ 
            = \int_{\Omega_{\xvec}} \!\!\! \big(\hat{\thetavecmicro}(\alphavecmicro) - \thetavecmicro_{*}^{0} \big) \big(\hat{\thetavecmicro}(\alphavecmicro) - \thetavecmicro_{*}^{0} \big)^{T} p(\alphavecmicro;\varthetavecmicro_{0}) \ud \alphavecmicro \nonumber \\ = {\text{\textbf{MSE}}} \big(\hat{\thetavecmicro}(\xvec),\thetavecmicro_{*}^{0} \big),\hspace{3.75cm}
        \eeqna
        where the first equality is obtained substituting \eqref{eq:point_wise_equ} and \eqref{pseudo1_pseudo2}.
        
        It should be noted that while Vuong also introduces an auxiliary class of distributions related to the true \ac{pdf} to derive the \ac{mcrb} (see \cite{Richmond_Fusion_2018,vuong1986cramer}), his construction is more restrictive. Specifically, the class is explicitly designed so that the pseudo-true parameter coincides with that of the original model, and the auxiliary score function is proportional to the assumed one in a neighborhood of the pseudo-true parameter. In contrast, our approach imposes only pointwise equivalence constraints, 
        where the pseudo-true parameter equivalence is obtained here as a byproduct.
        In addition, proportionality of the score functions is not assumed but instead emerges naturally as the optimizer over our class. Thus, in our framework, Vuong’s construction appears as a special case of a broader class, rather than an imposed assumption.

        \vspace{-0.2cm}
    \subsection{Naive Misspecified CRB for Pointwise Equivalent Models}
    \label{mcrb_general_derivation}
        Based on Subsection \ref{sub:Model_trans}, we can conclude that using $\tilde{p}(\xvec;\gammavecmicro_0)$ instead of $p(\xvec;\varthetavecmicro_0)$ defines a misspecified estimation model whose pseudo-true parameter is still $\thetavec_{*}^0$. Furthermore, since these \acp{pdf} coincide pointwisely and the pseudo-true parameters are equal, the bias and the \ac{mse} under $\tilde{p}(\xvec;\gammavecmicro_0)$ are identical to those under $p(\xvec; \varthetavecmicro_{0})$. Hence, any \ac{mse} bound derived under $\tilde{p}(\xvec;\gammavecmicro)$ can be used for the original model. 
        
        In this subsection, we investigate the use of the naive \ac{mcrb} for the pointwise equivalent models, i.e. using $\tilde{p}(\xvec; \gammavecmicro)$ as the ``true” \ac{pdf} \big(instead of $p(\xvec;\varthetavecmicro)$\big), while all other elements of the settings \big(e.g. the assumed \ac{pdf} $f(\xvec;\thetavecmicro)$\big) remain as described in Subsection \ref{model_subsec}. To this end, we first define the locally misspecified unbiasedness condition following Definition \ref{local_m_unbiasedness1}, under the pointwise equivalent \ac{pdf} $\tilde{p}(\xvec;\gammavecmicro)$. 
        \begin{definition}[Equivalent naive local misspecified unbiasedness]
        \label{def:missp_unbias}
             Let the assumed \ac{pdf} $f(\xvec;\thetavecmicro)$ satisfy Condition \ref{cond1}. An estimator $\hat{\thetavec}(\xvec)$ is a naive locally misspecified unbiased estimator of $\thetavec_{*}^{0}$ \ac{wrt} the pointwise equivalent \ac{pdf} $\tilde{p}(\xvec; \gammavecmicro)$ if 
             \vspace{-0.25cm}
            \begin{subequations}
                \beqna
                    \mathbb{E}_{\tilde{p}} \big[\hat{\thetavec}(\xvec) - \tilde{\thetavec}_{*}(\gammavecmicro_0) \big] \hspace{0.7cm} &=& \zerovec , \hspace{1.4cm}\label{pointwise_m_unbiasedness_1_tilde}
                \\
                    \nabla_{\gammavecsmall} \mathbb{E}_{\tilde{p}; \gammavecsmall} [\hat{\thetavec}(\xvec) - \tilde{\thetavec}_{*}(\gammavecmicro) ]\big|_{\gammavecsmall = \gammavecsmall_{0}} &=& \zerovec . \label{eq:local_miss_bias}
                \eeqna
            \end{subequations} 
        \end{definition}
        
        It can be seen that the condition in \eqref{pointwise_m_unbiasedness_1_tilde} is identical to the original unbiasedness condition in \eqref{pointwise_m_unbiasedness_1}, since the pointwise expectation \ac{wrt} $p(\xvec;\varthetavecmicro)$ and \ac{wrt} $\tilde{p}(\xvec;\gammavecmicro)$ are identical, and since according to \eqref{pseudo1_pseudo2} the associated pseudo-true parameters satisfy $\tilde{\thetavec}_{*}(\gammavecmicro_0)=\thetavec_{*}^{0}$. 
        However, the second condition in \eqref{eq:local_miss_bias} is different from \eqref{local_m_unbiasedness_1}. Similar to Claim \ref{claim_der_pt}, it can be shown that under Conditions \ref{cond1}-\ref{cond4},  \eqref{eq:local_miss_bias} can be reformulated as
        \beqna
        \label{17_tmp_tilde}
            {\mathbb{E}}_{\tilde{p}} [ (\hat{\thetavec}(\xvec)- \thetavec_{*}^{0} ) \nabla_{\gammavecsmall} \log \tilde{p}(\xvec;\gammavecmicro_{0}) ]
            \hspace{-0.07cm}=\hspace{-0.07cm} {\Amat}^{-1} ( \thetavecmicro_{*}^{0} ) \Bmat_{\tilde{p};f}( \thetavecmicro_{*}^{0},\gammavecmicro_{0} ),
        \eeqna  
        where $\Amat(\thetavecmicro) $ is defined in \eqref{eq:def_A_theta} and $\Bmat_{\tilde{p};f}(\thetavecmicro,\varthetavecmicro)$ is obtained by substituting $\tilde{p}(\xvec;\gammavecmicro)$ in \eqref{eq:def_B_p_f}.
        This result can be derived similarly to the derivation in Appendix \ref{App:der_PT} by substituting $\nabla_{\varthetavecsmall} \log p(\xvec; \varthetavecmicro_{0})$ with $\nabla_{\gammavecsmall} \log \tilde{p}(\xvec; \gammavecmicro_{0})$.
        The condition in \eqref{17_tmp_tilde} implies that local misspecified unbiasedness under $\tilde{p}(\xvec;\gammavecmicro)$ can be written based on the correlation between the estimation error, $\hat{\thetavecmicro}(\xvec) - \thetavecmicro_{*}^{0}$, and the score function $\nabla_{\gammavecsmall} \log \tilde{p}(\xvec;\gammavecmicro_{0})$, rather than the true score, $\nabla_{\varthetavecsmall} \log p(\xvec; \varthetavecmicro_{0})$ as in \eqref{17_tmp}.
        It is noted that, due to the pointwise equivalence in \eqref{eq:point_wise_equ}, all the expectations in \eqref{pointwise_m_unbiasedness_1_tilde}, \eqref{eq:local_miss_bias}, and \eqref{17_tmp_tilde} can be computed \ac{wrt} $p(\xvec;\varthetavecmicro_{0})$.

        The following proposition presents the naive \ac{mcrb} under a general pointwise equivalent model, $\tilde{p}(\xvec;\gammavecmicro)$.
        \begin{proposition}[Naive MCRB for pointwise equivalent models]
        \label{claim2_tmp}
            Let $\tilde{p}(\xvec;\gammavecmicro)$ be a pointwise equivalent \ac{pdf} (Definition \ref{def:point_wise_equ}) satisfying Conditions \ref{cond1}-\ref{cond4} together with the assumed \ac{pdf} $f(\xvec;\thetavecmicro)$. Consider an estimator $\hat{\thetavec}(\xvec)$ that is naive locally misspecified unbiased \ac{wrt} $\tilde{p}(\xvec;\gammavecmicro)$ (Definition \ref{def:missp_unbias}). Then,
            \beqna
            \label{eq:bias_CRB_p_tilde}
                {\text{\textbf{MSE}}} \big(\hat{\thetavecmicro}(\xvec),\thetavecmicro_{*}^{0} \big) \succeq {\text{\textbf{nMCRB}}}_{\tilde{p}}(\thetavecmicro_{*}^{0},\gammavecmicro_{0}) ,
            \eeqna
            where $ {\text{\textbf{nMCRB}}}_{\tilde{p}}(\cdot) $ is obtained by substituting $p=\tilde{p}$ and $\varthetavecmicro_{0}=\gammavecmicro_{0}$ in \eqref{nMCRB_def}.
            Similarly, equality in \eqref{eq:bias_CRB_p_tilde} is attained  {\em{iff}} \eqref{eq:MVU_bCRB} holds a.s.  after substituting $p=\tilde{p}$ and $\varthetavecmicro_{0}=\gammavecmicro_{0}$ in \eqref{eq:MVU_bCRB}. 
        \end{proposition}
        \vspace{-0.1cm}
        \begin{IEEEproof}
            The proof is along the lines of the proof of Theorem \ref{mcrb_theorem_5_terms}, where we replace $p(\xvec;\varthetavecmicro)$ with $\tilde{p}(\xvec;\gammavecmicro)$.
            Due to the pointwise equivalence from \eqref{def:point_wise_equ}, all the expectation operators can be computed \ac{wrt} $p(\xvec;\varthetavecmicro)$.
        \end{IEEEproof}

        The inequality in \eqref{eq:bias_CRB_p_tilde}  holds for any pointwise equivalent \ac{pdf} $\tilde{p}(\xvec;\gammavecmicro)$. 
        Thus, it provides a class of lower bounds on the \ac{mse}.
        Since $p(\xvec;\varthetavecmicro)$ and $\tilde{p}(\xvec;\gammavecmicro)$ coincide at the operating point, each bound in this class also applies to the original estimation problem described in Subsection~\ref{model_subsec}, i.e. they bound the \ac{mse} from \eqref{eq:miss_MSE}. Different choices of $\tilde{p}(\xvec;\gammavecmicro)$ generally yield different bounds. 
        In particular, the trivial choice of $\tilde{p}(\xvec;\gammavecmicro) = p(\xvec;\varthetavecmicro)$ results in 
        the naive \ac{mcrb} from Theorem \ref{mcrb_theorem_5_terms}, with its disadvantages described in Subsection \ref{limitation_Th2_subsection}. Other choices of $\tilde{p}(\xvec;\gammavecmicro)$ may lead to tighter bounds. 
        In the next subsection, we identify the tightest bound within this class.

\vspace{-0.25cm}
    \subsection{Tightest Bound in Pointwise Equivalent Models}
    \label{sub:opt_equi}
        In this subsection, we derive the tightest naive \ac{mcrb} that is obtained among all pointwise equivalent \acp{pdf} and show that it coincides with the classical \ac{mcrb}. Then, we provide a sufficient condition on $\tilde{p}(\xvec;\gammavecmicro)$ for achieving it.
        \vspace{-0.1cm}
        \begin{proposition}[Tightest naive MCRB]
        \label{theorem_tightest}
            Let  ${\mathcal{P}}_{\mathrm{pe}}$ denote the class of pointwise equivalent \acp{pdf} (Definition \ref{def:point_wise_equ}) satisfying Conditions \ref{cond1}-\ref{cond4} with the assumed \ac{pdf} $f(\xvec;\thetavecmicro)$. Then, the tightest naive \ac{mcrb} within this class is the \ac{mcrb} in \eqref{eq:def_MCRB}, i.e. 
            \beqna
            \label{eq:MCRB_sup}
                \text{\textbf{MCRB}}(\thetavecmicro_*^{0}) = \max_{\tilde{p}\in{\mathcal{P}}_{\mathrm{pe}}} \text{\textbf{nMCRB}}_{\tilde{p}}(\thetavecmicro_*^{0}, \gammavecmicro_{0}),
            \eeqna
            where the maximization is \ac{wrt} the Loewner order. Moreover, the maximum in \eqref{eq:MCRB_sup} is attained \eqref{nMCRB_def} iff  $p=\tilde{p}$, $\varthetavecmicro_{0}=\gammavecmicro_{0}$, and 
            \beqna
            \label{eq:lin_p_tilde}
                \Wmat \, {\nabla_{\gammavecsmall}}^{\!\!\! T} \log \tilde{p}(\xvec;\gammavecmicro_{0}) = {\nabla_{\thetavecsmall}}^{\!\!\! T} \log f(\xvec; \thetavecmicro_{*}^{0}),  ~~~ \text{a.s.},
            \eeqna
            where $\Wmat \in \mathbb{R}^{N_{2} \times N_{3}}$ is independent of $\xvec$.

        \end{proposition}
        
        \begin{IEEEproof}
            The proof appears in Appendix \ref{App:tightest_nMCRB}.
        \end{IEEEproof}

        This proposition implies that the \ac{mcrb} is a valid lower bound on the \ac{mse} of any estimator $\hat{\thetavec}(\xvec)$ that is naive locally misspecified unbiased \ac{wrt} equivalent models satisfying \eqref{eq:lin_p_tilde}, and that it is the tightest possible to obtain among pointwise equivalent models. A special case where the condition in \eqref{eq:lin_p_tilde} holds is for estimation problems that satisfy $ N_{1}=N_3$ and
        \beqna
        \label{eq:lin_assumed_true_pdf}
            \Wmat \, {\nabla_{\varthetavecsmall}}^{\!\!\! T} \log p(\xvec;\varthetavecmicro_{0}) = {\nabla_{\thetavecsmall}}^{\!\!\! T} \log f(\xvec; \thetavecmicro_{*}^{0}),~~~\forall \xvec \in \Omega_{\xvec}.
        \eeqna
         In this case, choosing $\tilde{p}(\xvec;\gammavecmicro) = p(\xvec;\varthetavecmicro)$,  we obtain immediately that \eqref{eq:lin_p_tilde} holds, and the naive \ac{mcrb} is reduced to the \ac{mcrb}.
        A simple example of such a model is given in Box \ref{box:example_Prop_Mod}.

        A natural question is whether there always exists a $\tilde{p}(\xvec;\gammavecmicro)$ that satisfies both the pointwise equivalence \eqref{eq:point_wise_equ} and the proportional-score condition \eqref{eq:lin_p_tilde}, thereby yielding the \ac{mcrb} as a lower bound. The following proposition provides a general method to construct such a \ac{pdf} for any model.

        \begin{figure}[ht]
            \refstepcounter{myboxcounter} 
            \begin{mybox}[label=box:example_Prop_Mod]{Example for Proportional Models}
                Consider the estimation problem where the true and assumed \acp{pdf} are
                \begin{boxsubequations}
                    \beqna 
                        p(\xvec; \vartheta) \hspace{-0.25cm} &=& \hspace{-0.25cm} \mathcal{N}(\xvec, \vartheta \onevec, \sigma_{1}^2 \Imat_{N}),
                        \\
                        f(\xvec; \theta) \hspace{-0.25cm} &=& \hspace{-0.25cm} \mathcal{N}(\xvec, \theta \onevec, \sigma_{2}^2 \Imat_{N}),
                    \eeqna
                \end{boxsubequations}
                
                \noindent where $\sigma_{1}^{2} \neq \sigma_{2}^{2}$. Substituting this model in \eqref{eq:pseudo_def_gen}, we obtain that  $\theta_{*}(\vartheta) = \vartheta$. Thus, it can be verified that
                \begin{boxequation}
                \label{example_2}
                    \hspace{-0.4cm}\nabla_{\! \vartheta} \log p(\xvec;\vartheta) \hspace{-0.05cm} = \hspace{-0.05cm} \frac{\sigma_{2}^{2}}{\sigma_{1}^{2}} \nabla_{\!\theta} \log f(\xvec; \theta_{*}) \hspace{-0.05cm} = \hspace{-0.05cm} \frac{ (\xvec - \vartheta \onevec)^{T} \onevec
                    }{\sigma_{1}^{2}}.
                \end{boxequation}
                
                \noindent This result implies that \eqref{eq:lin_assumed_true_pdf} is satisfied in this case for $\Wmat=\frac{\sigma_{1}^{2}}{\sigma_{2}^{2}}$. Thus,  in this case, by choosing $\tilde{p}(\xvec;\gamma) = p(\xvec;\vartheta)$,  we obtain that \eqref{eq:lin_p_tilde} holds, and the naive \ac{mcrb} is reduced to the \ac{mcrb}.
                To see that, we substitute this model in \eqref{eq:def_A_theta}, \eqref{eq:def_B_theta}, \eqref{eq:def_B_p_f}, and \eqref{J_def}, to obtain the information terms  
                \begin{boxsubequations}
                    \beqna
                        \label{A_box3}
                        A(\theta_{*})  = {N}/{\sigma_{2}^{2}},~~~ 
                        B(\theta_{*})  = {N \sigma_{1}^{2}}/{\sigma_{2}^{4}},
                        \\
                        \label{J_box3}
                        J_{p}(\vartheta)  =  {N}/{\sigma_{1}^{2}},~~~
                        B_{p;f}(\theta_{*}, \vartheta)  = {N}/{\sigma_{2}^{2}}. 
                    \eeqna
                \end{boxsubequations}
                
                \noindent By substituting \eqref{A_box3} 
                in \eqref{eq:def_MCRB}, we obtain that the \ac{mcrb} is \vspace{-0.3cm}
                \begin{boxequation}
                \label{mcrb_Box3}
                    \text{MCRB}(\theta_{*}^{0}) = {\sigma_{1}^{2}}/{N}.
                \end{boxequation}
                
                \noindent Similarly, using the expressions in \eqref{A_box3} and \eqref{J_box3},
                implies that the naive \ac{mcrb} from \eqref{nMCRB_def} is 
                \begin{boxequation}
                    {\text{nMCRB}}_{p}(\theta_{*}^{0},\vartheta_{0}) = {\sigma_{1}^{2}}/{N},
                \end{boxequation}
                
                \noindent i.e. coincides with \eqref{mcrb_Box3} as expected from Proposition \ref{theorem_tightest}. 
            \end{mybox} 
                        \vspace{-0.25cm}
        \end{figure}
        \begin{proposition}
        \label{claim7}
            Let $\Omega_{\gammavecsmall} = \Omega_{\thetavecsmall}$ and $\gammavecmicro_{0} = \thetavecmicro_{*}^{0}$. Consider the function
            \beqna
            \label{eq:gen_shape_ptilde}
                 \tilde{p}^{(o)}(\xvec;\gammavecmicro) = \frac{1}{c(\gammavecmicro)} g\bigg(\frac{f(\xvec; \gammavecmicro)}{f(\xvec; \thetavecmicro_{*}^{0})} \bigg) p(\xvec; \varthetavecmicro_{0}),~\gammavecmicro \in \Omega_{\thetavecsmall},
            \eeqna
            where 
            \beqna
            \label{c_def}
                c(\gammavecmicro)\define \int_{\Omega_{\xvec}} g\bigg(\frac{f(\alphavecmicro; \gammavecmicro)}{f(\alphavecmicro; \thetavecmicro_{*}^{0})} \bigg) p(\alphavecmicro; \varthetavecmicro_{0}) \ud \alphavecmicro
            \eeqna
            is a normalizing constant, and $g:\mathbb{R}_{+}\to\mathbb{R}_{+}$ is a twice-differentiable bijection  satisfying $g(1)=1$ and $g'(1) \neq 0$.
            Then, $\tilde{p}^{(o)}(\xvec;\gammavecmicro)$ has the following properties:
            {\renewcommand{\theenumi}{P.\arabic{enumi}}
            \begin{enumerate}
                \item\label{Prop1} It is a valid \ac{pdf} for any $\gammavecmicro \in \Omega_{\thetavecsmall}$.
                
                \item\label{Prop2} It is pointwise equivalent to $p(\xvec; \varthetavecmicro_{0})$ at $\gammavecmicro_0=\thetavecmicro_{*}^{0}$, assuming it satisfies  Conditions \ref{cond1}, \ref{cond2}, and \ref{cond4}.
                \item\label{Prop3} It satisfies the proportional-score condition in \eqref{eq:lin_p_tilde} for $\Wmat = \frac{1}{g'(1)} \Imat_{N_{2}}$, and therefore attains the tightest bound, i.e. the \ac{mcrb}, as characterized in Proposition \ref{theorem_tightest}.
            \end{enumerate}}
        \end{proposition}
          \vspace{-0.15cm}
        \begin{IEEEproof}
            The proof appears in Appendix \ref{App:P_tilde_proof}.
        \end{IEEEproof}

        It should be noted that since $f(\xvec;\thetavecmicro)$ is strictly positive on $\Omega_{\xvec}\times\Omega_{\thetavecsmall}$ (see after \eqref{eq:pseudo_def}), the \ac{pdf} $\tilde{p}^{(o)}(\xvec; \gammavecmicro)$ in \eqref{eq:gen_shape_ptilde} is well-defined for all $\gammavecmicro\in\Omega_{\thetavecsmall}$ and inherits the same support as $p(\xvec;\varthetavecmicro)$.
        In fact, the \ac{pdf} $\tilde{p}^{(o)}(\xvec;\gammavec)$ is a reweighted version of the true model $p(\xvec;\varthetavecmicro_{0})$, where the weights are determined by the likelihood ratio, $\tfrac{f(\xvec;\gammavecsmall)}{f(\xvec;\thetavecsmall_{*}^{0})}$, which measures how well the observation vector $\xvec$ fits a general parameter $\gammavecmicro$ relative to the pseudo-true parameter $\gammavecmicro_0=\thetavecmicro_{*}^{0}$. The likelihood ratio is then transformed by the function $g(\cdot)$ to yield $\tilde{p}^{(o)}(\xvec; \gammavecmicro)$.  
        Here, the role of the factor $g\bigl(\tfrac{f(\xvec;\gammavecsmall)}{f(\xvec;\thetavecsmall_{*}^{0})}\bigr)$ is twofold: it ensures that the pointwise equivalence holds at the operating point, and enforces the proportional-score condition in \eqref{eq:lin_p_tilde}.
        Examples of possible $\tilde{p}^{(o)}(\xvec; \gammavecmicro)$ functions are given in Box \ref{box:example_P_tilde}.
        
        From a performance-bound perspective, $\tilde{p}^{(o)}(\xvec;\gammavecmicro)$ can be viewed as the least informative distribution within the class of pointwise equivalent models, as it maximizes the attainable \ac{mse} lower bound while imposing pointwise constraints. 
        In this sense, $\tilde{p}^{(o)}(\xvec;\gammavecmicro)$ acts as an information envelope around the true model, preserving only pointwise information while discarding higher-order structure that misspecified estimators cannot utilize. 
        Unlike Vuong’s approach, where score proportionality is imposed by construction, Proposition \ref{theorem_tightest} shows that such proportionality emerges naturally as an optimality condition of a supremum over pointwise equivalent models.

        \begin{figure}
            \refstepcounter{myboxcounter} 
            \begin{mybox}[label=box:example_P_tilde]{\hspace{-0.05cm}Examples of Auxiliary Functions from Proposition \hspace{-0.04cm}\ref{claim7}}
                \textit{\textbf{Example 1:}} 
                    By taking $g(z) = z$, $\forall z>0$, we obtain
                    \begin{boxequation}
                        \tilde{p}^{(o)}(\xvec; \gamma) = \frac{1}{c(\gamma)} \frac{f(\xvec; \gamma)}{f(\xvec; \theta_{*}^{0})} p(\xvec; \vartheta_{0}),
                    \end{boxequation}
                    
                    \noindent where $c(\gamma)$ is the normalization factor from \eqref{c_def}.
                    
                \textbf{\textit{Example 2:}}
                    \cite{vuong1986cramer} By taking $g(z)=1+\exp{(1-z)}$, $z>0$, one obtains
                    \begin{boxequation}
                        \hspace{-0.29cm} \tilde{p}^{(o)}(\xvec; \gamma) = \frac{1}{c(\gamma)}\! \Big[\! 1 +  \exp{\!\Big( 1 - \frac{f(\xvec;\gamma)}{f(\xvec;\theta_{*}^{0})}\Big)}\!\Big] p(\xvec; \vartheta_{0}),
                    \end{boxequation}
                    
                    \noindent where $c(\gamma)$ is the normalization factor from \eqref{c_def}. 
            \end{mybox} 
        \end{figure}
\section{Revised MCRB Theorem}
\label{sub:Final_MCRB}

    In this section, we formalize the class of estimators for which the \ac{mcrb} holds (Subsection \ref{new_class_subsec}) and provide a revised \ac{mcrb} theorem to replace Conjecture \ref{instead_of_theorem} (Subsection \ref{revised_theorem_subsection}), based on the results from the previous section.
    Then, we define the concept of a misspecified efficiency, and establish its relations with the \ac{mml} estimator (Subsection \ref{efficent_subsection}).

\vspace{-0.25cm}
    \subsection{Class of estimators} 
    \label{new_class_subsec}
        We now identify the class of estimators for which the  \ac{mcrb} is a valid lower bound.  By Proposition \ref{theorem_tightest}, the \ac{mcrb} is the maximum of the naive bound over the class of pointwise equivalent models. Thus, the relevant class of estimators is obtained by requiring naive local misspecified unbiasedness 
         \ac{wrt} a pointwise equivalent \ac{pdf} that attains this maximum.

         More precisely, let $\tilde{p}^{(o)}(\xvec;\gammavecmicro)$ be a pointwise equivalent \ac{pdf} that attains the maximum in \eqref{eq:MCRB_sup}. Then, according to Definition \ref{def:missp_unbias}, the associated estimator class is characterized by naive local misspecified unbiasedness \ac{wrt} $\tilde{p}^{(o)}$, i.e., by  \eqref{pointwise_m_unbiasedness_1_tilde} and \eqref{eq:local_miss_bias} with $\tilde{p}^{(o)}$. 
         Proposition \ref{theorem_tightest} further shows that every such optimizer satisfies the proportional-score condition \eqref{eq:lin_p_tilde}. Hence, the local unbiasedness condition expressed through the auxiliary score $\nabla_{\gammavecsmall}\log \tilde{p}^{(o)}(\xvec;\gammavecmicro_0)$ can be rewritten equivalently in terms of the assumed score $\nabla_{\thetavecsmall}\log f(\xvec;\thetavecmicro_*^0)$. 
         As shown in Appendix \ref{appendix_newMCRB}, this yields an equivalent definition of unbiasedness that depends only on the true and assumed \acp{pdf}, $p$ and $f$, and does not explicitly involve $\tilde{p}^{(o)}(\xvec;\gammavecmicro)$.
        \begin{definition}[Revised misspecified local unbiasedness]
        \label{new_miss_unbiasdness}
            Let the assumed \ac{pdf} $f(\xvec;\thetavecmicro)$ satisfy Conditions \ref{cond1}-\ref{cond4}, and let the true \ac{pdf} be $p(\xvec;\varthetavecmicro)$. An estimator $\hat{\thetavec}(\xvec)$ is a {\em{locally misspecified-unbiased}} estimator of $\thetavecmicro_{*}^{0}$ at $\varthetavecmicro = \varthetavecmicro_{0}$ if 
            \begin{subequations}
                \beqna
                    \mathbb{E}_{p} [ \hat{\thetavec}(\xvec) - \thetavec_{*}^{0} ] &=& \zerovec , \label{new_miss_unbiasdness_cond1}\hspace{3cm}
                \\
                     \mathbb{E}_{p} [ ( \hat{\thetavec}(\xvec) - \thetavec_{*}^{0} ) \nabla_{\thetavecsmall} \log f(\xvec; \thetavecmicro_{*}^{0}) ]  &=& {\Amat}^{-1}(\thetavecmicro_{*}^{0}) \Bmat(\thetavecmicro_{*}^{0}) . \label{new_miss_unbiasdness_cond2}
                \eeqna
            \end{subequations} 
        \end{definition}

        Definition \ref{new_miss_unbiasdness} can be interpreted as the optimizer-induced version of naive local misspecified unbiasedness.
        In particular, the pointwise unbiasedness condition in \eqref{new_miss_unbiasdness_cond1} is identical to the original pointwise unbiasedness in \eqref{pointwise_m_unbiasedness_1}. The condition \eqref{new_miss_unbiasdness_cond2} is the \emph{first-order} (local) unbiasedness condition obtained after eliminating the auxiliary optimizer $\tilde{p}^{(o)}$ using the proportional-score relation in \eqref{eq:lin_p_tilde}. Thus, \eqref{new_miss_unbiasdness_cond2} expresses the relevant local constraint through the assumed score $\nabla_{\thetavecsmall}\log f(\xvec;\thetavecmicro_*^0)$ rather than the true score in \eqref{17_tmp}.
        This reformulation is important for two reasons.  First, it identifies the estimator class associated with the maximum characterization of the \ac{mcrb} in Proposition~\ref{theorem_tightest}. Second, it yields a condition stated solely in terms of the standard misspecified information matrices, without explicit reference to the auxiliary \ac{pdf} $\tilde{p}^{(o)}$. 
        In the correctly specified case ($f=p$ and $\thetavec_*^{0}=\varthetavec_0$),
        \eqref{new_miss_unbiasdness_cond2} reduces to the standard local unbiasedness condition of the \ac{crb}.

    \subsection{Revised MCRB theorem}
    \label{revised_theorem_subsection}
        \begin{theorem}[The revised misspecified Cram$\acute{\text{e}}$r-Rao bound]
        \label{thm:Our_MCRB}
            Let the assumed \ac{pdf} $f(\xvec;\thetavecmicro)$ satisfy Conditions \ref{cond1}-\ref{cond4}, and $\hat{\thetavec}(\xvec)$ be a locally misspecified unbiased estimator as defined in Definition \ref{new_miss_unbiasdness}. Then,
            \beqna
            \label{eq:finally_our_bound}
                {\text{\textbf{MSE}}} \big(\hat{\thetavecmicro}(\xvec),\thetavecmicro_{*}^{0} \big) \succeq \text{\textbf{MCRB}}(\thetavecmicro_{*}^{0}),
            \eeqna
            where ${\text{\textbf{MCRB}}}(\thetavecmicro_{*}^{0})$ is the classical \ac{mcrb} from 
            \eqref{eq:def_MCRB}. 
            Equality is obtained \textit{iff} 
            \beqna
            \label{eq:MVU}
                \hat{\thetavec}(\xvec) = \thetavec_{*}^{0} + {\Amat}^{-1}(\thetavecmicro_{*}^{0}) \,\, {\nabla_{\thetavecsmall}}^{\!\!\! T} \! \log f(\xvec;\thetavecmicro_{*}^{0}),~~a.s.
            \eeqna
        \end{theorem}

        \begin{IEEEproof} 
            The proof appears in Appendix \ref{appendix_newMCRB}.
        \end{IEEEproof}
        In Box \ref{box:counter_example_cont}, we apply the new \ac{mcrb} theorem on the example in Box \ref{box:counter_example} and Box \ref{box:counter_example_cont_1}.

        Similar to the conventional derivation of the \ac{crb} \cite{kay}, the bound in \eqref{eq:finally_our_bound} can also be derived directly via the Cauchy-Schwarz inequality, using the auxiliary function $\nabla_{\thetavecsmall} \log f(\xvec; \thetavecmicro_{*}^{0}) $, and applying the new unbiasedness conditions, \eqref{new_miss_unbiasdness_cond1} and \eqref{new_miss_unbiasdness_cond2}. Alternatively, it can be derived by solving a constrained minimization problem with these conditions.

        The equality condition in \eqref{eq:MVU} requires that the estimation error lie in the closed linear span of the assumed score, $\nabla_{\thetavecsmall}\log f(\xvec;\thetavecmicro_{*}^{0})$, with coefficient matrix uniquely determined as $\Amat^{-1}(\thetavecmicro_{*}^{0})$. 
        The estimator defined in \eqref{eq:MVU} satisfies the revised unbiasedness conditions \eqref{new_miss_unbiasdness_cond1}--\eqref{new_miss_unbiasdness_cond2}, and attains the \ac{mcrb}.
        In general, however, the expression in \eqref{eq:MVU} does not define an implementable  estimator since it depends explicitly on the pseudo-true parameter $\thetavecmicro_{*}^{0}$, similar to the equality condition in the classical \ac{crb} \cite[p. 44-45]{kay} in correctly specified models. Nevertheless,  in certain cases efficient estimators do exist, as discussed below. 
        Moreover, \eqref{eq:MVU} provides a constructive characterization of the efficient direction and can be used to develop Fisher-scoring-type schemes for computing the \ac{mml} estimator under misspecification, thereby linking the equality condition to practical estimation algorithms.

            \vspace{-0.25cm}
    \subsection{Efficient Estimation Under the revised MCRB theorem}
    \label{efficent_subsection}

        Having identified the class of estimators associated with the \ac{mcrb}, we now characterize the estimator that attains the bound in the misspecified setting.
        This leads to the following revised definition of efficiency in misspecified models.
        \begin{definition}[Misspecified efficient estimator]
        \label{def:eff_est}
            Let $\Omega_{\varthetavecsmall}^{\circ}\subseteq\Omega_{\varthetavecsmall}$ be an open set. 
            An estimator $\hat{\thetavecmicro}_{\mathrm{eff}}(\xvec)$ is a uniformly misspecified efficient estimator  on $\Omega_{\varthetavecsmall}^{\circ}$ if, for every $\varthetavecmicro\in\Omega_{\varthetavecsmall}^{\circ}$, of $\thetavecmicro_{*}^{0}$ at $\varthetavecmicro=\varthetavecmicro_{0}$ if
            (i) it is a locally misspecified unbiased estimator of $\thetavec_{*}(\varthetavecmicro)$ 
            in the sense of Definition \ref{new_miss_unbiasdness}, and (ii) it achieves the corresponding  \ac{mcrb}, i.e., 
            \beqna
                \text{\textbf{MSE}} \big(\hat{\thetavecmicro}_{\mathrm{eff}}(\xvec),\thetavecmicro_{*}(\varthetavecmicro)\big) = \text{\textbf{MCRB}} \big(\thetavecmicro_{*}(\varthetavecmicro)\big), ~ \forall \, \varthetavecmicro \in \Omega_{\varthetavecsmall}^{\circ}.
            \eeqna
        \end{definition}
        In the correctly specified case, if an efficient estimator exists, then it is attained by the \ac{ml} estimator \cite{kay}. The definition above provides the corresponding notion under model misspecification, where efficiency is required to hold {\em{uniformly}} over an open set of operating points, enabling a non-asymptotic characterization analogous to the classical case.

        Under model misspecification, an analogous role is played by the \ac{mml} estimator, as formalized in the following theorem.

        \begin{theorem}[Misspecified efficiency and the \ac{mml} estimator]
        \label{thm:eff_est}
            Let the assumed \ac{pdf} $f(\xvec;\thetavecmicro)$ satisfy Conditions \ref{cond1}-\ref{cond4}. 
            Assume that there exists an open set $\Omega_{\varthetavecsmall}^{\circ}\subseteq\Omega_{\varthetavecsmall}$ such that 
            the image $\thetavec_{*}(\Omega_{\varthetavecsmall}^{\circ})$ contains a nonempty open set 
            $\Omega_{\thetavecsmall}^{\circ}\subseteq\Omega_{\thetavecsmall}$. 
            If a uniformly misspecified efficient estimator $\hat{\thetavecmicro}_{\mathrm{eff}}(\xvec)$ exists on $\Omega_{\varthetavecsmall}^{\circ}$, then
            for every $\xvec$ such that $\hat{\thetavecmicro}_{\mathrm{eff}}(\xvec)\in\Omega_{\thetavecsmall}^{\circ}$, we have
            \beqna
            \label{eq:MML_eff}
                \hat{\thetavec}_{\mathrm{eff}}(\xvec) = \hat{\thetavec}_{\text{MML}}(\xvec).
            \eeqna
        \end{theorem}
        \begin{IEEEproof}
            According to Definition \ref{def:eff_est}, $\hat{\thetavec}_{\mathrm{eff}}(\xvec)$ is a misspecified unbiased estimator that achieves the \ac{mcrb}. Hence, the \ac{mcrb} equality condition in \eqref{eq:MVU}, implies that for any $\varthetavecmicro \in \Omega_{\varthetavecsmall}^{\circ}$,
            \beqna
                \hat{\thetavec}_{\mathrm{eff}}(\xvec) = \thetavec_{*}(\varthetavecmicro) + \Amat^{-1}\!\big(\thetavecmicro_{*}(\varthetavecmicro)\big) \nabla_{\thetavecsmall}^{T} \log f\big(\xvec;\thetavecmicro_{*}(\varthetavecmicro)\big).
            \eeqna
            By assumption, $\thetavec_{*}(\Omega_{\varthetavecsmall}^{\circ})=\Omega_{\thetavecsmall}^{\circ}$. 
            Therefore, the equality above holds for all $\thetavec \in \Omega_{\thetavecsmall}^{\circ}$, i.e.,
            \beqna
            \label{eq:MVU_eff_theta}
                \hat{\thetavec}_{\mathrm{eff}}(\xvec) = \thetavec + \Amat^{-1}(\thetavec) \nabla_{\thetavecsmall}^{T} \log f(\xvec;\thetavec), \quad \forall \thetavec \in \Omega_{\thetavecsmall}^{\circ}.
            \eeqna
            Now, for every $\xvec$ such that $\hat{\thetavecmicro}_{\mathrm{eff}}(\xvec)\in\Omega_{\thetavecsmall}^{\circ}$, substituting $\thetavec = \hat{\thetavecmicro}_{\mathrm{eff}}(\xvec)$ in \eqref{eq:MVU_eff_theta} yields
            \beqna
                \nabla_{\thetavecsmall} \log f\big(\xvec;\hat{\thetavecmicro}_{\mathrm{eff}}(\xvec)\big) = \zerovec.
            \eeqna
            Thus, $\hat{\thetavecmicro}_{\mathrm{eff}}(\xvec)$ satisfies the likelihood equation and is therefore an \ac{mml} estimator, $\hat{\thetavecmicro}_{\text{MML}}(\xvec)$.
        \end{IEEEproof}

        Theorem \ref{thm:eff_est} provides the misspecified analogue of the classical relation between the \ac{ml} estimator and efficiency \ac{wrt} the \ac{crb} under correct specification. The result relies on a uniform equality condition over an open set. In contrast, the well-known asymptotic efficiency of the \ac{mml} estimator \cite{white1982maximum} is inherently local, depending only on the behavior of the likelihood in a neighborhood of the pseudo-true parameter.

        \begin{figure}
            \refstepcounter{myboxcounter} 
            \begin{mybox}[label=box:counter_example_cont]{Illustration of the Revised MCRB Estimator Class}
                Consider the model from Box \ref{box:counter_example}. By substituting \eqref{A_sc}-\eqref{B_sc} in \eqref{new_miss_unbiasdness_cond2}, we obtain that the requirement for a locally misspecified unbiased estimator, $\hat{\theta}(\xvec)$, from Definition \ref{new_miss_unbiasdness} is 
                \vspace{-0.2cm}
                \begin{boxequation}
                \label{ex1_cond2}
                    \hspace{-0.23cm} \mathbb{E}_{p}\bigg[ (\hat{\theta}(\xvec) - \theta_{*}^{0}) \nabla_{\! \theta} \! \log f(\xvec; \theta_{*}^{0}) \bigg] = \frac{\varepsilon + (N - 1) \sigma^{2}}{N \sigma^{2}}.
                \end{boxequation}
                The single-sample estimator $\hat{\theta}_{1}(\xvec) = x_{1}$ has a lower \ac{mse} than the \ac{mcrb}. This can now be explained by the fact that it does not satisfy the unbiasedness condition in \eqref{ex1_cond2}, i.e.
                \vspace{-0.2cm}
                \begin{boxequation}
                    \hspace{-0.3cm} \mathbb{E}_{p} \Big[ (x_{1} - \theta_{*}^{0}) \nabla_{\! \theta} \! \log f(\xvec;\theta_{*}^{0}) \Big] \! = \! \frac{\varepsilon}{\sigma^{2}} \! \neq \! A^{-1}(\theta_{*}^{0}) B(\theta_{*}^{0}). \hspace{-0.15cm}
                \end{boxequation}
                \noindent In contrast, the \ac{mml} estimator $\hat{\theta}_{\text{MML}}(\xvec) = \frac{1}{N} \onevec^{T} \xvec$ is a locally misspecified unbiased estimator, since
                \begin{boxequation}
                    \hspace{-0.35cm} \mathbb{E}_{p} \Big[ \Big(\frac{\onevec^{T} \xvec}{N} - \theta_{*}^{0} \Big) \nabla_{\! \theta} \! \log f(\xvec;\theta_{*}^{0}) \Big] = \frac{\varepsilon + (N - 1) \sigma^{2}}{N \sigma^{2}}. \hspace{-0.15cm}
                \end{boxequation}
                The \ac{mse} of the \ac{mml} estimator satisfies
                \begin{boxequation}
                    \text{MSE}(\hat{\theta}_{\text{MML}}(\xvec), \theta_{*}^{0}) = \frac{1}{N^{2}} \onevec^{T} \Sigmamat \onevec = \text{MCRB}(\theta_{*}^{0}),
                \end{boxequation}

                \noindent where the last equality stems from \eqref{mcrb_ex1}.
                This result can be explained now by the fact that Definition \ref{new_miss_unbiasdness} characterizes the class of misspecified estimators bounded by the \ac{mcrb}.

                 In Fig. \ref{fig:misspecified_bias_gaussian_example}, we present the empirical biases of $\hat{\theta}_{1}(\xvec)$ and the \ac{mml} estimator based on Definition \ref{new_miss_unbiasdness}. 
                 That is, we present the ``regular bias", $\mathbb{E}_{p} \big[ \hat{\theta}(\xvec) - \theta_{*}^{0} \big]$, and the ``score bias", $\mathbb{E}_{p} \big[ \big( \hat{\theta}(\xvec) - \theta_{*}^{0} \big) \nabla_{\theta} \log f(\xvec; \theta_{*}^{0}) \big] - {A}^{-1}(\theta_{*}^{0}) B(\theta_{*}^{0})$. It can be seen that the \ac{mml} estimator satisfies both conditions of Definition \ref{new_miss_unbiasdness}, whereas the single-sample estimator $\hat{\theta}_{1}(\xvec)$ violates the score-unbiasedness condition from \eqref{new_miss_unbiasdness_cond2}, and therefore lies outside the class relevant for the \ac{mcrb}. This explains the results presented in Fig. \ref{fig:misspecified_gaussian_example}.
                    
                    \centering
                    \includegraphics[width=0.7\linewidth]{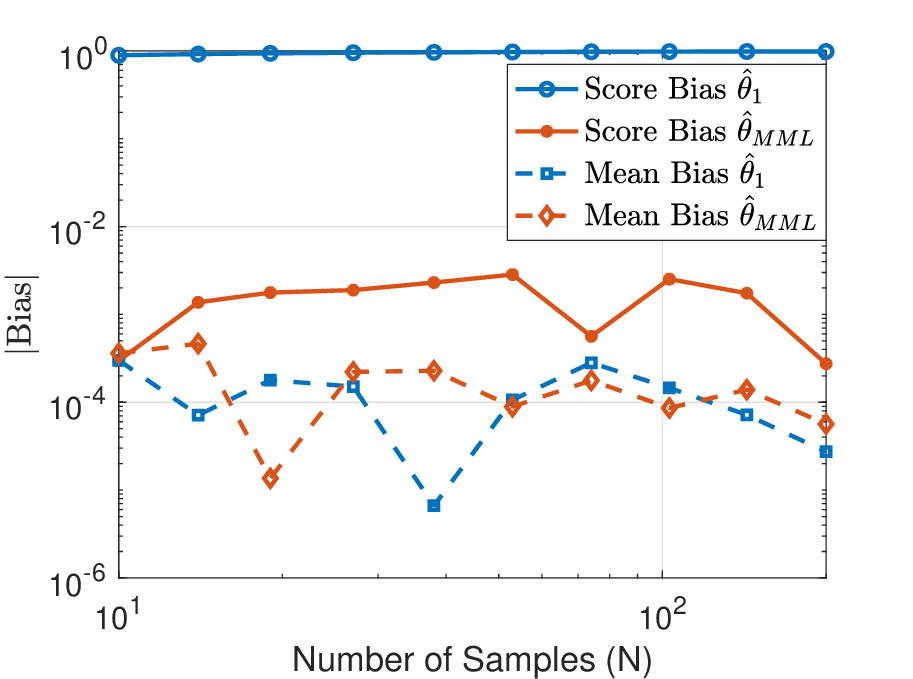}
                    \caption{Bias for the example in Box~\ref{box:counter_example}. }
                    \label{fig:misspecified_bias_gaussian_example}
                  \end{mybox} 
        \end{figure}   

\section{Simulations}
    In the following simulations, we illustrate the theoretical results and examine the behavior of the proposed framework under model misspecification for \ac{doa} estimation.
    We consider a planar array of $M$ sensors that receives a narrowband signal from a single far-field source located at an unknown deterministic angle, $\phi\in \Omega_{\phi}$, relative to the array. The received signal is modeled as
    \be
        \xvec = \avec(\phi) s + \vvec 
    \ee
    where $s\in \mathbb{C} $ is an unknown deterministic complex signal, and $\avec: \Omega_{\phi} \rightarrow \mathbb{C}^M$ is the steering vector.  
    We define the real parameter vector $\varthetavecmicro = [\phi,s_R,s_I]^T$, where $s_R$ and $s_I$ are the real and imaginary parts of $s$,  respectively.
    The additive noise, $\vvec$ is modeled as a circularly symmetric complex Gaussian vector with zero mean and true covariance matrix $\bsigma$, while the estimator assumes a covariance matrix $\sigma^2\Imat_{M}$. 
    That is, model misspecification arises from assuming spatially white noise instead of spatially colored noise, as frequently occurs in the presence of interference.
    Thus, the assumed and true \acp{pdf} are 
    \[p(\xvec;\varthetavecmicro)={\cal CN}(\xvec,\avec(\phi)s,\bsigma), ~~f(\xvec;\thetavecmicro)={\cal CN}(\xvec,\avec(\phi)s,\sigma^2\Imat_M).
    \]

    {\bf{MCRB and unbiasedness:}}
    In this setting, it can be shown that the pseudo-true parameter vector from \eqref{eq:pseudo_def} coincides with the true parameters, i.e. $\thetavecmicro_* = \varthetavecmicro$.
    For brevity, we define the following shorthand notation:
    \beqna
        F_{11} &= \dot{\avec}^H \Sigmamat \dot{\avec}, \quad
        F_{12} = \avec^H \Sigmamat \dot{\avec}, \quad
        F_{22} = \avec^H \Sigmamat \avec.
    \eeqna
    Substitution of this model in 
    \eqref{eq:def_A_theta} and \eqref{eq:def_B_theta} results is
    \beqna
        \Amat(\thetavecmicro) = \frac{2 M}{\Tr(\bsigma)} {\text{diag}}\left( \|\dot{\avec}\|^2|s|^2, M, M\right)
    \eeqna
    and
    \beqna
        \Bmat(\thetavecmicro)  = \frac{2 M^2}{\Tr^{2}(\bsigma)} \hspace{-0.1cm}
        \scalebox{0.97}{$ \begin{bmatrix}
            F_{11} |s|^2               & \text{Re}\{F_{12} s \} & - \text{Im}\{ F_{12}^{*} s\} \\
            \text{Re}\{F_{12} s\}        & F_{22}                 & 0                            \\
            -\text{Im}\{ F_{12}^{*} s \} & 0                      & F_{22}
        \end{bmatrix} $}, \hspace{-0.06cm}
    \eeqna
    where $\dot{\avec} \define  \nabla_{\phi} \avec(\phi)$. Therefore, the \ac{mcrb} is given by 
    \beqna
        \text{\textbf{MCRB}}(\thetavecmicro) = \frac{1}{2 M \|\dot{\mathbf{a}}\|^2 |s|^2} \nonumber \hspace{4.35cm} \\ \times
        \begin{bmatrix}
        \frac{M F_{11}}{\|\dot{\mathbf{a}}\|^2} & \text{Re}\{ F_{12} s \} & - \text{Im}\{ F_{12}^{*} s \} \\
        {\text{Re}}\{ F_{12} s \} & \frac{F_{22} \|\dot{\mathbf{a}}\|^2 |s|^2}{M} & 0 \\
        - \text{Im}\{ F_{12}^{*} s \} & 0 & \frac{F_{22} \|\dot{\mathbf{a}}\|^2 |s|^2}{M}
        \end{bmatrix}.
    \eeqna

    {\bf{Estimators:}}
        The \ac{mml} estimator under the assumed white-noise model and the ``oracle'' \ac{ml} estimator that uses the true covariance $\Sigmamat$ are given by:
        \begin{subequations}
            \beqna
                \hat{\phi}_{\text{MML}} \hspace{-0.2cm} &=& \hspace{-0.2cm} \arg\max_{\phi} \left|\avec^H(\phi)\xvec\right|^2, \label{eq:phi_MML}
                \\
                \hat{\phi}_{\text{ML}} &=& \hspace{-0.2cm} \arg\max_{\phi} \frac{\left|\avec^H(\phi){\Sigmamat}^{-1} \xvec \right|^2}{\avec^H(\phi) {\Sigmamat}^{-1}\avec(\phi)}. \label{eq:phi_OML}
            \eeqna
        \end{subequations}

   {\bf{Results:}} In \cref{DOA_vs_rho} we present simulation results for a uniform linear array with $M=8$ sensors, $s=e^{j\frac{\pi}{4}}$, and $\phi=\frac{\pi}{8}$. The true covariance is a Toeplitz matrix with entries $\bsigma_{i,j}=\sigma^2\rho^{|i-j|}$, where the real-value parameter $|\rho|<1$ controls the spatial correlation. For $\rho=0$, $\bsigma=\sigma^2\Imat_M$, and the assumed model is correctly specified. As $\rho$ increases, we have a stronger model misspecification. We set $\sigma^2=0.1$, which is a high \ac{snr} regime, so that the estimators are nearly pointwise unbiased in the sense of \eqref{new_miss_unbiasdness_cond1}.

In Fig. \ref{DOA_vs_rho}\textcolor{red}{a} we present the revised local misspecified-unbiasedness conditions for the estimator $\hat{\phi}$ from  \eqref{new_miss_unbiasdness_cond2}, i.e.  
\[\mathbb{E}_{p} [ ( \hat{\phi}- \phi ) \nabla_{\phi} \log f(\xvec; \thetavecmicro_{*}^{0})  ]  -\left[{\Amat}^{-1}(\thetavecmicro_{*}^{0}) \Bmat(\thetavecmicro_{*}^{0})\right]_{1,1}\] versus the correlation coefficient, $\rho$, for the two estimators. This figure shows that the \ac{mml} estimator satisfies the unbiasedness condition \eqref{new_miss_unbiasdness_cond2} for any $\rho$ and is therefore a member of the class of estimators for which the \ac{mcrb} is valid, as stated in \Cref{revised_theorem_subsection}.
In contrast, the oracle \ac{ml} estimator has a nonzero bias for $\rho>0$, and thus, is not a member of this class. 
   Fig. \ref{DOA_vs_rho}\textcolor{red}{b} shows the \ac{mse} of the \ac{mml} and \ac{ml} estimators for $\phi$ versus $\rho$, together with the \ac{mcrb} and the classical \ac{crb}, which is also the naive \ac{mcrb} in this case. 
Although the pseudo-true parameter coincides with the true parameter in this setting and the \ac{ml} estimator satisfies the conditions of \cref{instead_of_theorem}, Fig. \ref{DOA_vs_rho}\textcolor{red}{b} demonstrates that the \ac{mcrb} does not provide a valid lower bound for this estimator under misspecification, even asymptotically.

\vspace{-0.5cm}
    \begin{figure}[hbt]
        \centering
        \includegraphics[width=0.8\linewidth]{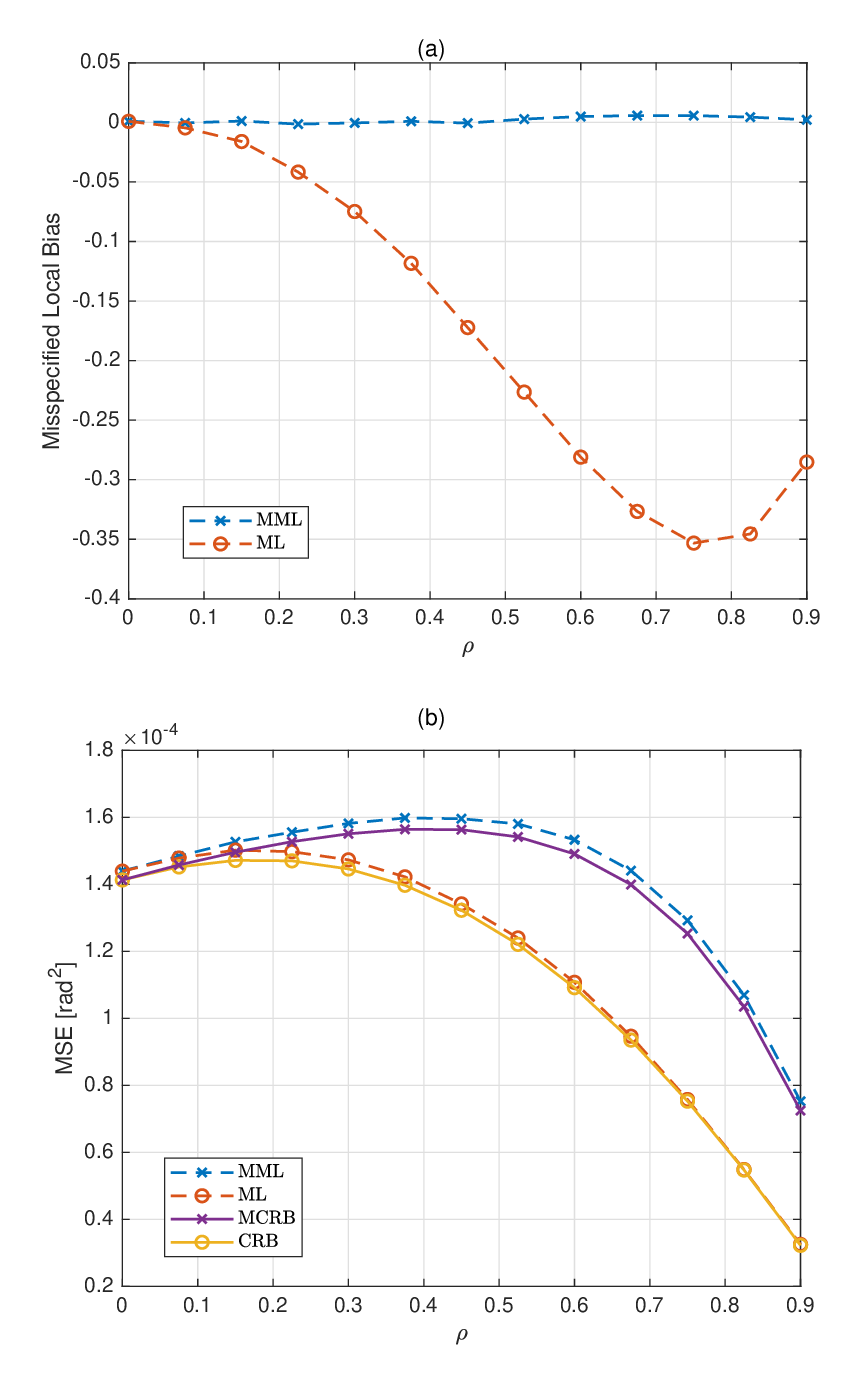}
        \caption{The (a) misspecified local bias  and (b) RMSE  versus the correlation coefficient, $\rho$.}
    \label{DOA_vs_rho}
    \end{figure} 

 \vspace{-0.5cm}
\section{Conclusions}
\label{Conclusions}
 \vspace{-0.05cm}
    This paper revisited the \ac{mcrb} and addressed fundamental gaps in its theoretical foundation. By introducing a principled notion of local misspecified unbiasedness and a family of pointwise equivalent models, we derived the \ac{mcrb} through an estimator-agnostic procedure that clarifies both its attainability and its structural properties. The resulting theory explicitly identifies the class of estimators for which the \ac{mcrb} is valid, and establishes a clear notion of efficiency under model misspecification, as well as its relation with the \ac{mml} estimator.
    Beyond providing a unified derivation, the proposed framework offers insight into the limitations of naive \ac{mcrb} and shows that the \ac{mcrb} can be interpreted as the tightest achievable naive \ac{mcrb}  with minimal information about the true model. Future research directions include extensions to cyclic \cite{Khatib_cyclicMCRB} and constrained \cite{Nitzan_Routtenberg_Tabrikian_2023} parameter spaces,  and development of large error bounds, such as the Barankin bound.

\appendices
\renewcommand{\thesection}{\Alph{section}}  

\newcommand{\customappendix}[1]{
    \refstepcounter{section} 
    \section*{Appendix \thesection{}: #1} 
}
    \vspace{-0.2cm}
    \customappendix{Proof of Claim \ref{claim_der_pt}}
    \label{App:der_PT}
        By using the definition of the expectation operator, Condition \ref{cond4}, and the chain rule, it can be verified that
        \vspace{-0.1cm}
        \beqna
        \label{integrals_app}
           \nabla_{\varthetavecsmall} \mathbb{E}_{p; \varthetavecsmall} \big[ \hat{\thetavec}(\xvec) - \thetavec_{*}(\varthetavecmicro) \big] = - \int_{\Omega_{\xvec}} \! (\nabla_{\varthetavecsmall} \thetavecmicro_{*}(\varthetavecmicro) ) \, p(\alphavecmicro; \varthetavecmicro) \ud \alphavecmicro \nonumber \\ + \int_{\Omega_{\xvec}} \! \big(\hat{\thetavec}(\alphavecmicro) - \thetavec_{*}(\varthetavecmicro) \big) \, \nabla_{\varthetavecsmall} p(\alphavecmicro; \varthetavecmicro) \ud \alphavecmicro \hspace{2.63cm} \nonumber \\= - \nabla_{\varthetavecsmall} \thetavecmicro_{*}(\varthetavecmicro) + {\mathbb{E}}_{p; \varthetavecsmall} [ (\hat{\thetavec}(\xvec) - \thetavec_{*}(\varthetavecmicro) ) \, \nabla_{\varthetavecsmall} \log p(\xvec; \varthetavecmicro) ],
        \eeqna
        where the last equality is obtained by using the expectation operator definition,
        and the facts that $\nabla_{\varthetavecsmall} \log p(\xvec; \varthetavec) =\frac{1}{p(\xvec; \varthetavecsmall)} \nabla_{\varthetavecsmall} p(\xvec; \varthetavec)$, and $\nabla_{\varthetavecsmall} \thetavecmicro_{*}(\varthetavecmicro)$ is a deterministic parameter. 
        By substituting the derivation from \eqref{integrals_app} in \eqref{local_m_unbiasedness_1}, we obtain that the requirement in \eqref{local_m_unbiasedness_1} can be written as the condition
        \beqna
        \label{new_condition}
             {\mathbb{E}}_{p} [ (\hat{\thetavec}(\xvec)- \thetavec_{*}(\varthetavec) ) \, \nabla_{\varthetavecsmall} \log p(\xvec; \varthetavecmicro_{0}) ] =  \nabla_{\varthetavecsmall} \thetavecmicro_{*}(\varthetavecmicro_{0}).
        \eeqna
        In \cite{vuong1986cramer},  it is shown that under Conditions \ref{cond1}-\ref{cond4}
        \beqna
        \label{eq:der_PT}
             \nabla_{\varthetavecsmall}\thetavec_{*}(\varthetavecmicro) = {\Amat}^{-1} \big( \thetavecmicro_{*} \big) \Bmat_{p;f}(\thetavecmicro_{*},\varthetavecmicro),
        \eeqna
        where $\Amat( \cdot) $ and $\Bmat_{p;f}(\cdot,\cdot)$ are defined in \eqref{eq:def_A_theta} and \eqref{eq:def_B_p_f}, respectively.
        By substituting \eqref{eq:der_PT} in \eqref{new_condition}, we obtain \eqref{17_tmp}.
    \vspace{-0.25cm}
    \customappendix{Proof of Theorem \ref{mcrb_theorem_5_terms}}
    \label{App:naive_mcrb_theorem}

        The derivation of \ac{mse} lower bounds can be done using the Cauchy-Schwarz inequality. Another way is to solve a constrained minimization problem of the \ac{mse} under specific unbiasedness conditions (see, e.g. \cite{PCRB_J, fauss2021variational,menni2011new}). Here, we adopt the second approach to establish the bound.
  
        The \ac{mse} minimization under the constraints in \eqref{pointwise_m_unbiasedness_1} and \eqref{local_m_unbiasedness_1} can be written as the following optimization problem 
        \begin{flalign}
            \min_{\hat{\thetavecsmall}(\xvec) \in \Omega_{\thetavectiny}} \avec^{T} \, \mathbb{E}_{p} [ (\hat{\thetavec}(\xvec) - \thetavec_{*}^{0} ) (\hat{\thetavec}(\xvec) - \thetavec_{*}^{0} )^T  ] \, \avec \hspace{-1.2cm} \label{eq:MCRB_opt} \\ 
            \customlabel{Con0_MCRB}{$\tilde{\text{C}}$.1} \hspace{0.1cm} \scalebox{1}{$\mathbb{E}_{p} [\hat{\thetavec}(\xvec) - \thetavec_{*}^{0} ]$} \hspace{1.6cm} &=&  \scalebox{1}{$\zerovec,$} \hspace{1.3cm} \notag \tikz[remember picture] \node[coordinate,yshift=1em, xshift=-20.7em] (n1) {}; \\ 
            \customlabel{Con1_MCRB}{$\tilde{\text{C}}$.2} \hspace{0.8cm} \scalebox{1}{$\nabla_{\varthetavecsmall} \mathbb{E}_{p; \varthetavecsmall} [\hat{\thetavec}(\xvec) - \thetavec_{*}(\varthetavec) ]\big|_{\varthetavecsmall=\varthetavecsmall_0} $} \hspace{0.2cm} &=& \scalebox{1}{$\zerovec,$} \hspace{1.3cm} \notag  \tikz[remember picture] \node[coordinate,yshift=-0.9em] (n2) {};
            \begin{tikzpicture}[thick,overlay,remember picture]
                    \path (n2) -| node[coordinate] (n3) {} (n1);
                     \draw[decorate,decoration={brace,mirror,amplitude=5pt}]
                            (n1) -- (n3) node[midway, left=4pt] {{s.t.}};
            \end{tikzpicture}
        \end{flalign}
        where $ \avec \in \mathbb{R}^{N_{2}}$ is an arbitrary vector\footnote{The use of $\avec$ is in order to characterize a matrix-valued lower bound via scalarization, where a valid matrix lower bound is obtained only if the minimizing estimator of \eqref{eq:MCRB_opt} is independent of the choice of $\avec$.}.
        According to Claim \ref{claim_der_pt}, under Conditions \ref{cond1}-\ref{cond4}, the constraint \eqref{Con1_MCRB} can be replaced with \eqref{17_tmp}. Using this replacement, the Lagrangian for the optimization problem in \eqref{eq:MCRB_opt} is
        \beqna
        \label{eq:Lagrangian}
            L(\hat{\thetavecmicro}(\xvec), \lambdavec, \Lambdamat) \hspace{-0.25cm} &=& \hspace{-0.25cm} \avec^{T} \mathbb{E}_{p} \big[ (\hat{\thetavecmicro}(\xvec) - \thetavecmicro_{*}^{0} ) (\hat{\thetavecmicro}(\xvec) - \thetavecmicro_{*}^{0} )^{T} \big] \avec \nonumber \hspace{1.2cm} \\ &-& \hspace{-0.25cm} 2 \, \lambdavec^{T} \big(\mathbb{E}_{p} [\hat{\thetavecmicro}(\xvec) ] - \thetavecmicro_{*}^{0} \big) \nonumber \\ &-& \hspace{-0.25cm} 2 \, \text{Tr} \Big\{ \Lambdamat^{T} \big( \mathbb{E}_{p} \big[ \hat{\thetavecmicro} (\xvec) \, \nabla_{\varthetavecsmall} \log p(\xvec;\varthetavecmicro_{0}) \big] \nonumber \\ &-& \hspace{-0.25cm} {\Amat}^{-1}(\thetavecmicro_{*}^{0}) \Bmat_{p;f}(\thetavecmicro_{*}^{0}, \varthetavecmicro_{0}) \big) \big\}
            ,
        \eeqna
        where $2 \, \lambdavec$ and $2 \, \Lambdamat$ are the Lagrangian multipliers associated with Constraints \eqref{Con0_MCRB} and  \eqref{Con1_MCRB}, respectively. We aim to minimize the Lagrangian \ac{wrt} the estimator $\hat{\thetavecmicro}(\xvec)$ at a specific point $\xvec \in \Omega_\xvec$. 
        This minimization of \eqref{eq:Lagrangian} \ac{wrt} $\hat{\thetavecmicro}(\xvec)$ yields 
        \beqna
        \label{est_opt}
            \avec \avec^{T} \big({\hat{\thetavecmicro}}^{(o)}\!(\xvec) - \thetavecmicro_{*}^{0} \big) = \lambdavec + \Lambdamat \, {\nabla_{\varthetavecsmall}}^{\!\!\! T} \log p(\xvec;\varthetavecmicro_{0}),
        \eeqna
            
        The Lagrange multipliers can be found by enforcing the problem constraints on the estimator in \eqref{est_opt}. First, by  taking the expectation of \eqref{est_opt} \ac{wrt}  $p(\xvec;\varthetavecmicro_0)$, one obtains
        \beqna
        \label{cond_on_lambda1}
            \avec \avec^{T} \mathbb{E}_{p} [ \hat{\thetavecmicro}^{(o)}\!(\xvec) - \thetavecmicro_{*}^{0} ] \! = \! \lambdavec \! + \! \Lambdamat \, \mathbb{E}_{p}[ {\nabla_{\varthetavecsmall}}^{\!\!\! T} \! \log p(\xvec;\varthetavecmicro_{0}) ] = \lambdavec,
        \eeqna
        where the last equality is obtained from
        \beqna
        \label{eq:reg_cond_p}
            \mathbb{E}_{p}[ \nabla_{\varthetavecsmall} \log p(\xvec;\varthetavecmicro_{0}) ] = \zerovec,
        \eeqna
        which is a result of Condition \ref{cond4}, i.e. the regularity of $p(\xvec;\vartheta)$. Substituting the pointwise unbiasedness \eqref{Con0_MCRB} in \eqref{cond_on_lambda1} and rearranging,  we obtain $\lambdavec = \zerovec$, implying that the first constraint \eqref{Con0_MCRB} is redundant (similarly to in the case of the conventional \ac{crb}). By substituting $\lambdavec = \zerovec$ in \eqref{est_opt}, we obtain that the estimator that solves the optimization in \eqref{eq:MCRB_opt} satisfies
        \beqna
        \label{est_opt2}
            \avec \avec^{T} \big(\hat{\thetavecmicro}^{(o)}\!(\xvec) - \thetavecmicro_{*}^{0} \big) = \Lambdamat \, {\nabla_{\varthetavecsmall}}^{\!\!\! T} \log p(\xvec;\varthetavecmicro_{0}).
        \eeqna
        Right multiplying \eqref{est_opt2} by $\nabla_{\varthetavecsmall} \log p(\xvec;\varthetavecmicro_{0}) $ and applying expectation \ac{wrt} $p(\xvec;\varthetavecmicro_0)$ results in 
        \beqna
        \label{stage1}
            \avec \avec^{T} \mathbb{E}_{p} [ (\hat{\thetavecmicro}^{(o)}\!(\xvec) - \thetavecmicro_{*}^{0} ) \, \nabla_{\varthetavecsmall} \log p(\xvec;\varthetavecmicro_{0}) ] \hspace{2.15cm} \nonumber \\ = \Lambdamat \, \mathbb{E}_{p} [ {\nabla_{\varthetavecsmall}}^{\!\!\! T} \log p(\xvec;\varthetavecmicro_{0}) \nabla_{\varthetavecsmall} \log p(\xvec;\varthetavecmicro_{0}) ] =  \Lambdamat \,  \Jmat_{p}(\varthetavecmicro_{0}), 
        \eeqna
        where the first equality is obtained from \eqref{eq:reg_cond_p}, and $\Jmat_{p}(\varthetavecmicro)$ is defined in \eqref{J_def}. Now, by substituting the local unbiasedness \eqref{Con1_MCRB} (as expressed in \eqref{17_tmp}) in \eqref{stage1} and rearranging, we obtain 
        \beqna
        \label{stage2}
            \Lambdamat = \avec \avec^{T} {\Amat}^{-1}(\thetavecmicro_{*}^{0}) \Bmat_{p;f}(\thetavecmicro_{*}^{0},\varthetavecmicro_{0}) \Jmat_{p}^{-1}(\varthetavecmicro_{0}).
        \eeqna
        By substituting \eqref{stage2} in \eqref{est_opt2}, we obtain that the estimator that solves the optimization problem in \eqref{eq:MCRB_opt} is given by
        \beqna
        \label{est_opt3}
            \avec \avec^{T} \big(\hat{\thetavecmicro}^{(o)}\!(\xvec) - \thetavecmicro_{*}^{0} \big) \nonumber \hspace{5.5cm} \\ = \avec \avec^{T} {\Amat}^{-1}(\thetavecmicro_{*}^{0}) \Bmat_{p;f}(\thetavecmicro_{*}^{0},\varthetavecmicro_{0}) \Jmat_{p}^{-1}(\varthetavecmicro_{0}) \, {\nabla_{\varthetavecsmall}}^{\!\!\! T} \log p(\xvec;\varthetavecmicro_{0}).
        \eeqna
        To obtain a matrix-valued bound (valid for all $\avec\in \mathbb{R}^{N_{2}}$), we require a single estimator that satisfies \eqref{est_opt3}, $\forall \avec \in \mathbb{R}^{N_{2}}$. 
        This is possible {\em{iff}} the equality condition in \eqref{eq:MVU_bCRB} is satisfied.
        The \ac{mse} of the estimator in \eqref{eq:MVU_bCRB} gives the minimum \ac{mse} for the constrained minimization problem:
        \beqna
        \label{eq:b-CRB}
            \mathbb{E}_{p} [ (\hat{\thetavec}^{(o)}\!(\xvec) -\thetavec_{*}^{0} ) (\hat{\thetavec}^{(o)}\!(\xvec) -\thetavec_{*}^{0} )^{T} ] \hspace{2.8cm} \nonumber \\= {\Amat}^{-1}(\thetavecmicro_{*}^{0}) \Bmat_{p;f}(\thetavecmicro_{*}^{0},\varthetavecmicro_{0}) \Jmat_{p}^{-1}(\varthetavecmicro_{0}) \Bmat_{p;f}^{T}(\thetavecmicro_{*}^{0},\varthetavecmicro_{0}) {\Amat}^{-1}(\thetavecmicro_{*}^{0}).
        \eeqna
        Since this is the minimum \ac{mse} for a locally misspecified unbiased estimator, we obtain the bound in \eqref{eq:MSE_inq_MCRB2}-\eqref{nMCRB_def} on the class of locally misspecified unbiased estimators.
        \vspace{-0.25cm}
    \customappendix{Proof of Claim \ref{cl:MCRB_vs_bCRB}}
    \label{App:CS-MCRB-bCRB}
        According to the Cauchy-Schwarz inequality,
        \beqna 
        \label{Cs1}
            {\mathbb{E}}_{p} [\phivec(\xvec) \phivec^{T}(\xvec) ] \nonumber \hspace{5.5cm} \\ \succeq {\mathbb{E}}_{p} [\phivec(\xvec) \etavec^{T}(\xvec) ] {\mathbb{E}}_{p}^{-1} [ \etavec(\xvec) \etavec^{T}(\xvec) ] {\mathbb{E}}_{p} [\etavec(\xvec) \phivec^{T}(\xvec) ],
        \eeqna
        for any measurable functions $\phivec : \Omega_{\xvec} \to \mathbb{R}^{d_{1}}$ and $\etavec : \Omega_{\xvec} \to \mathbb{R}^{d_{2}}$, 
        provided that ${\mathbb{E}}_{p}[\etavec(\xvec)\etavec^T(\xvec)]$ is a nonsingular matrix.
        Equality is obtained {\em{iff}} 
        \beqna
        \label{eq_CS}
            \phivec(\xvec) = \Wmat \, \etavec(\xvec),  ~~~\forall \xvec \in \Omega_{\xvec},
        \eeqna
        where $\Wmat \in \mathbb{R}^{d_1 \times d_2}$ is independent of $\xvec$ and the equality holds {\em{almost surely}} \ac{wrt} $p(\xvec)$.
        Substituting $\phivecmicro(\xvec) = {\nabla_{\thetavecsmall}}^{\!\!\! T} \log f(\xvec;\thetavecmicro_{*}^{0})$ and $\etavecmicro(\xvec) = {\nabla_{\varthetavecsmall}}^{\!\!\! T} \log p(\xvec;\varthetavecmicro_{0})$ in \eqref{Cs1}, one obtains
        \beqna
        \label{Cs2}
            \Bmat(\thetavecmicro_{*}^{0}) \succeq  \Bmat_{p;f}(\thetavecmicro_{*}^{0},\varthetavecmicro_{0}) \Jmat_{p}^{-1}(\varthetavecmicro_{0}) \Bmat_{p;f}^{T}(\thetavecmicro_{*}^{0},\varthetavecmicro_{0}),
        \eeqna
        where $ \Bmat(\cdot)$, $\Bmat_{p;f}(\cdot,\cdot)$, and $\Jmat_{p}(\cdot)$ are defined in \eqref{eq:def_B_theta}, \eqref{eq:def_B_p_f}, and \eqref{J_def}, respectively. Now, due to Condition \ref{cond3}, we can multiply \eqref{Cs2} from left and right by ${\Amat}^{-1}(\thetavecmicro_{*}^{0})$ to obtain  Claim \ref{cl:MCRB_vs_bCRB}, i.e. 
        \beqna
            \text{\textbf{MCRB}}(\thetavecmicro_{*}^{0}) \succeq \text{\textbf{nMCRB}}_{p}(\thetavecmicro_{*}^{0}, \varthetavecmicro_{0}),
        \eeqna
        where we used the definitions of the bounds in  \eqref{eq:def_MCRB} and \eqref{nMCRB_def}.

    \vspace{-0.25cm}
    \customappendix{Proof of Proposition \ref{theorem_tightest}}
    \label{App:tightest_nMCRB}
        The first part of this proof is along the lines of the proof in Appendix \ref{App:CS-MCRB-bCRB}, where we replace $p(\xvec;\varthetavecmicro)$ with $ \tilde{p}(\xvec;\gammavecmicro)$.
        In particular, substituting $\phivecmicro(\xvec) = {\nabla_{\thetavecsmall}}^{\!\!\! T} \log f(\xvec;\thetavecmicro_{*}^{0})$ and $\etavecmicro(\xvec) = {\nabla_{\gammavecsmall}}^{\!\!\! T} \log \tilde{p}(\xvec;\gammavecmicro_{0})$ in \eqref{Cs1}-\eqref{eq_CS}, we get the relation 
        \beqna
        \label{eq:MSE_inq_MCRB2_Claim5}
            \text{\textbf{MCRB}}(\thetavecmicro_{*}^{0}) \succeq \text{\textbf{nMCRB}}_{\tilde{p}}(\thetavecmicro_{*}^{0}, \varthetavecmicro_{0}),
        \eeqna
        where equality is obtained in \eqref{eq:MSE_inq_MCRB2_Claim5} {\em{iff}} the relation in \eqref{eq:lin_p_tilde} is satisfied (in the almost surely sense). 
       
    \vspace{-0.25cm}
    \customappendix{Proof of Proposition \ref{claim7}}
    \label{App:P_tilde_proof}
    
        \noindent\textbf{P.1 (Valid PDF).} Since $g:\mathbb{R}_{+}\to\mathbb{R}_{+}$ and $p(\xvec;\varthetavecmicro_0) \geq 0$,  it follows from \eqref{eq:gen_shape_ptilde} that $\tilde{p}^{(o)}(\xvec;\gammavecmicro)\ge 0$. Moreover, $c(\gammavecmicro)>0$ because $g>0$ on $\mathbb{R}_{+}$ and $p(\xvec;\varthetavecmicro_0)$ is a \ac{pdf}. Finally, by the normalization factor in \eqref{c_def}, we have $\int_{\Omega_{\xvec}} \tilde{p}^{(o)}(\alphavecmicro;\gammavecmicro)\, \ud \alphavecmicro=1$. Hence, $\tilde{p}^{(o)}(\cdot;\gammavecmicro)$ is a valid \ac{pdf} $\forall \gammavecmicro$, i.e. Property \ref{Prop1} holds.

        \noindent\textbf{P.2 (Pointwise equivalence at $\gammavecmicro_0$).} At $\gammavecmicro_{0} = \thetavecmicro_{*}^{0}$, we have $\frac{f(\xvec; \gammavecsmall_0)}{f(\xvec; \thetavecsmall_{*}^{0})} =1$ and since $g(1) = 1$, \eqref{eq:gen_shape_ptilde} is reduced to
        \beqna
            \tilde{p}^{(o)}(\xvec;\gammavecmicro_{0}) = \frac{1}{c(\gammavecmicro_{0})}  \, p(\xvec; \varthetavecmicro_{0}), ~ \forall \xvec \in \Omega_{\xvec}.
        \eeqna
        Consequently, the normalization factor from \eqref{c_def} satisfies  $c(\gammavecmicro_{0})=1$,  and therefore
        we obtain the pointwise equivalent condition in \eqref{eq:point_wise_equ}, which proves  Property \ref{Prop2}.

        \medskip
        \noindent\textbf{P.3 (Proportional-score condition).}

        Finally, using the definition in \eqref{eq:gen_shape_ptilde}, it can be verified that
        \begin{align}
        \label{eq:der_p_tilde}
            \nabla_{\gammavecsmall} \log \tilde{p}^{(o)}(\xvec;\gammavecmicro) = -\nabla_{\gammavecsmall}\log c(\gammavecmicro) + \frac{g'\!\big(r(\xvec;\gammavecmicro)\big)}{g\!\big(r(\xvec;\gammavecmicro)\big)} \, \nabla_{\gammavecsmall} r(\xvec;\gammavecmicro),
        \end{align}
        where
        $r(\xvec;\gammavecmicro)\triangleq \frac{f(\xvec;\gammavecmicro)}{f(\xvec;\thetavecsmall_*^0)}  >0$. 
        Because $f(\xvec;\thetavecmicro)$ is strictly positive, we have
        \beqna
        \label{r_def}
            \nabla_{\gammavecsmall} r(\xvec;\gammavecmicro) = r(\xvec;\gammavecmicro)\,\nabla_{\gammavecsmall}\log f(\xvec;\gammavecmicro).
        \eeqna
        In addition, using  \eqref{c_def}, it can be verified that
        \beqna
        \label{95}
            \hspace{-0.3cm} \nabla_{\gammavecsmall} \log \hspace{-0.33cm} &c& \hspace{-0.39cm} (\gammavecmicro) = \frac{1}{c(\gammavecmicro)}\int_{\Omega_{\xvec}} g'\!\big(r(\xvec;\gammavecmicro)\big)\,\nabla_{\gammavecsmall} r(\xvec;\gammavecmicro)\, p(\xvec;\varthetavecmicro_0)\, \ud \xvec \nonumber \\ &=& \hspace{-0.28cm} \frac{1}{c(\gammavecmicro)} \mathbb{E}_{p}\!\left[ g'\!\big(r(\xvec;\gammavecmicro)\big)\, r(\xvec;\gammavecmicro)\,\nabla_{\gammavecsmall}\log f(\xvec;\gammavecmicro) \right],
        \eeqna
        where the first equality is obtained assuming the standard interchange condition that permits differentiation under the integral in $c(\gammavecmicro)$ (satisfied under Conditions \ref{cond2} and \ref{cond4}), and the last equality is obtained by substituting \eqref{r_def}.
        Evaluating \eqref{95} at $\gammavecmicro_0 = \thetavecmicro_*^0$ gives $r(\xvec;\gammavecmicro_0) = 1$, $c(\gammavecmicro_0) = 1$, hence
        \beqna
        \label{eq:grad_log_c_at0}
            \nabla_{\gammavecsmall}\log c(\gammavecmicro_0) = g'(1)\,\mathbb{E}_{p}\!\left[\nabla_{\gammavecsmall}\log f(\xvec;\gammavecmicro_0)\right].
        \eeqna
        
        Since $\gamma_0 = \theta_*^{0}$, $\Omega_\gamma = \Omega_\theta$, and $f(x;\gamma) = f(x;\theta)$, the definition of the pseudo-true in \eqref{def:psuedo-true} (under the usual regularity assumptions) implies $\mathbb{E}_{p}\!\left[\nabla_{\gammavecsmall}\log f(\xvec;\gammavecmicro_0)\right] = \zerovec. $
        Thus, $\nabla_{\gammavecsmall}\log c(\gammavecmicro_0)=\zerovec$ by \eqref{eq:grad_log_c_at0}.
        Plugging this result and \eqref{r_def} into \eqref{eq:der_p_tilde} at $\gammavecmicro_0$ (and using $g(1)=1$ and $r(\xvec;\gammavecmicro_0)=1$) yields
        \beqna
            \nabla_{\gammavecsmall}\log \tilde{p}^{(o)}(\xvec;\gammavecmicro_0) = g'(1)\,\nabla_{\gammavecsmall}\log f(\xvec;\gammavecmicro_0).
        \eeqna
 
        This result implies that if we take $\tilde{p}^{(o)}(\xvec;\gammavecmicro)$ to be the pointwise equivalent \ac{pdf}  $ \tilde{p}(\xvec;\gammavecmicro)$, we obtain that the relation in \eqref{eq:lin_p_tilde} is satisfied at $\gammavecmicro_{0}=\thetavecmicro_{*}^{0}$ with $\Wmat = \frac{1}{g'(1)} \Imat_{N_{2}}$ (i.e. Property \ref{Prop3} holds).
        By Proposition \ref{theorem_tightest}, this implies that $\tilde{p}^{(o)}$ attains the tightest bound, i.e., the \ac{mcrb}.
    \vspace{-0.25cm}
    \customappendix{Proof of Theorem \ref{thm:Our_MCRB}}
    \label{appendix_newMCRB}
Theorem \ref{thm:Our_MCRB} is proved by substituting the optimal pointwise equivalent \ac{pdf} that satisfies \eqref{eq:lin_p_tilde} into Proposition \ref{claim2_tmp}.  For $\gammavecmicro_0 = \thetavecmicro_{*}^{0}$ and $\Wmat = \frac{1}{g'(1)} \Imat_{N_{2}}$ (see Property \ref{Prop3}), and by substituting \eqref{eq:lin_p_tilde} in \eqref{eq:def_B_p_f} and \eqref{J_def}, we obtain 
        \begin{subequations}
            \beqna
                \Bmat_{\tilde{p},f}(\thetavecmicro_{*}^{0},\gammavecmicro_{0}) &=& g'(1) \,\,\,  \Bmat(\thetavecmicro_{*}^{0}), \label{eq:def_B_p_tilde_f_psi}
            \\
                \Jmat_{\tilde{p}^{(o)}}(\gammavecmicro_{0}) &=& g'(1)^{2} \, \Bmat(\thetavecmicro_{*}^{0}), \label{eq:J_psi}
            \eeqna
        \end{subequations}   
        where $ \Bmat(\thetavecmicro) $ is defined in \eqref{eq:def_B_theta}.
        
        \noindent\textbf{Revised Unbiasedness:} A naive misspecified estimator \ac{wrt} $\tilde{p}^{(o)}(\xvec; \gammavecmicro)$ as defined in Definition \ref{def:missp_unbias} satisfies \eqref{pointwise_m_unbiasedness_1_tilde} and \eqref{17_tmp_tilde}.
        It can be verified that the pointwise condition in \eqref{new_miss_unbiasdness_cond1} is obtained by substituting $\tilde{p}^{(o)}(\xvec; \gammavecmicro)$ from \eqref{eq:gen_shape_ptilde} in \eqref{pointwise_m_unbiasedness_1_tilde}. By substituting \eqref{eq:lin_p_tilde} and \eqref{eq:def_B_p_tilde_f_psi} in \eqref{17_tmp_tilde} we obtain the condition in \eqref{new_miss_unbiasdness_cond2}.
        Thus, we can conclude that 
        Definition \ref{new_miss_unbiasdness} is obtained by substituting the optimal pointwise equivalent \ac{pdf}, i.e. $\tilde{p}^{(o)}(\xvec; \gammavecmicro)$, in the naive local misspecified unbiasedness from Definition \ref{def:missp_unbias}. 
        
        \noindent\textbf{Revised Theorem:} Replacing the unbiasedness requirements in Proposition \ref{claim2_tmp} with Definition \ref{new_miss_unbiasdness}, we obtain that for a locally misspecified unbiased estimator
        \beqna
        \label{eq:nMCRB_psi}
            {\text{\textbf{MSE}}} \big(\hat{\thetavecmicro}(\xvec),\thetavecmicro_{*}^{0} \big) \succeq {\text{\textbf{nMCRB}}}_{\tilde{p}^{(o)}}(\thetavecmicro_{*}^{0}, \gammavecmicro_{0}).
        \eeqna
        According to Proposition \ref{theorem_tightest}, we obtain
        \beqna
        \label{eq:nMBCRB_psi_MCRB}
            {\text{\textbf{nMCRB}}}_{\tilde{p}^{(o)}}(\thetavecmicro_{*}^{0}, \gammavecmicro_{0}) = {\text{\textbf{MCRB}}}(\thetavecmicro_{*}^{0}).
        \eeqna
        Substituting \eqref{eq:nMBCRB_psi_MCRB} in \eqref{eq:nMCRB_psi}, we obtain \eqref{eq:finally_our_bound}. By substituting \eqref{eq:lin_p_tilde}, \eqref{eq:def_B_p_tilde_f_psi}, and \eqref{eq:J_psi} in \eqref{eq:bias_CRB_p_tilde} we obtain the equality condition \eqref{eq:MVU}.
\vspace{-0.25cm}

\end{document}